\documentclass[10pt]{article}
\usepackage{amssymb}
\usepackage{geometry}
\usepackage{amsmath}
\usepackage{graphicx}
\usepackage{subcaption}

\usepackage{amsthm}
\usepackage{hyperref}
\usepackage{hhline}
\usepackage{wrapfig}
\usepackage{enumerate}
\usepackage{verbatim}
\usepackage{float}
\usepackage{color}
\usepackage{enumitem} 
\graphicspath{{converted_graphics/}}
\newtheorem{theorem}{Theorem}[section]
\newtheorem{corollary}[theorem]{Corollary}
\newtheorem{lemma}[theorem]{Lemma}

\newtheorem{proposition}[theorem]{Proposition}

\theoremstyle{definition}
\newtheorem{definition}[theorem]{Definition}

\newtheorem{remark}[theorem]{Remark}

\newtheorem{notation}[theorem]{Notation}

\newcommand{\bole}{\Box\hspace{-0.330cm}\vdash}
\newcommand{\bori}{\Box\hspace{-0.350cm}\dashv}
\newcommand{\A}{\mathcal{A}}
\newcommand{\E}{\mathcal{E}}
\newcommand{\NC}{\mathcal{NC}}
\newcommand{\Hom}{\operatorname{Hom}}

\usepackage[utf8]{inputenc}

\newcommand{\bb}{\mathbb{B}}
\newcommand{\cc}{\mathbb{C}}

\renewcommand{\ggg}{\mathbb{G}}

\renewcommand{\AA}{\mathcal{A}}

\newcommand{\CC}{\mathcal{C}}
\newcommand{\DD}{\mathcal{D}}
\newcommand{\EE}{\mathcal{E}}

\newcommand{\II}{\mathcal{I}}

\newcommand{\LL}{\mathcal{L}}
\newcommand{\MM}{\mathcal{M}}
\newcommand{\NN}{\mathcal{N}}

\newcommand{\PP}{\mathcal{P}}

\newcommand{\pai}{\left(}
\newcommand{\pad}{\right)}

\hypersetup{
    colorlinks=true,
    linkcolor=blue,
    filecolor=magenta,      
    urlcolor=cyan,
    pdftitle={Borrador},
    bookmarks=true,
    pdfpagemode=FullScreen,
}

\begin{document}

\title{Relations between infinitesimal non-commutative cumulants}

\author{A.~Celestino$\!\!\phantom{i}^{\ast}$, K.~Ebrahimi-Fard\footnote{
Department of Mathematical Sciences, Norwegian University of Science and Technology (NTNU), 7491 Trondheim, Norway. \texttt{adrian.celestino@ntnu.no}, \texttt{kurusch.ebrahimi-fard@ntnu.no}}, D.~Perales Anaya\footnote{Department of Pure Mathematics, University of Waterloo, Ontario, Canada. \texttt{dperales@uwaterloo.ca} }}

\date{ }

{\maketitle}

\begin{abstract}
Boolean, free and monotone cumulants as well as relations among them, have proven to be important in the study of non-commutative probability theory. Quite notably, Boolean cumulants were successfully used to study free infinite divisibility via the Boolean Bercovici--Pata bijection. On the other hand, in recent years the concept of infinitesimal non-commutative probability has been developed, together with the notion of infinitesimal cumulants which can be useful in the context of combinatorial questions.

In this paper, we show that the known relations among free, Boolean and monotone cumulants still hold in the infinitesimal framework. Our approach is based on the use of algebra of Grassmann numbers. Formulas involving infinitesimal cumulants can be obtained by applying a formal derivation to known formulas.

The relations between the various types of cumulants turn out to be captured via the shuffle algebra approach to moment-cumulant relations in non-commutative probability theory. In this formulation, (free, Boolean and monotone) cumulants are represented as elements of the Lie algebra of infinitesimal characters over a particular combinatorial Hopf algebra. The latter consists of the graded connected double tensor algebra defined over a non-commutative probability space and is neither commutative nor cocommutative. In this note it is shown how the shuffle algebra approach naturally extends to the notion of infinitesimal non-commutative probability space. The basic step consists in replacing the base field as target space of linear Hopf algebra maps by the algebra of Grassmann numbers defined over the base field. We also consider the infinitesimal analog of the Boolean Bercovici--Pata map. 
\end{abstract}

\tableofcontents


\section{Introduction}
\label{sec:intro}

Since Voiculescu introduced the theory of free probability \cite{Voi} together with the corresponding notion of free independence, other kinds of non-commutative probability theories have been studied, such as Boolean independence \cite{SW} and monotone independence \cite{Mur1}.
  A fundamental tool to attack problems in free probability is the notion of free cumulants, introduced by Speicher \cite{Spe} as a combinatorial way to study and use free independence. The analogue notions of Boolean cumulants \cite{SW} and monotone cumulants \cite{HS1,HS2} have been developed, sharing many common features. In particular, the combinatorics of cumulants shows close analogies between the different types of non-commutative probabilities. While free independence can be captured by the lattice of non-crossing partitions, Boolean independence makes use of interval partitions whereas monotone independence can be captured by monotone non-crossing partitions. 

Recently, in a series of papers, Ebrahimi-Fard and Patras \cite{EP1,EP2,EP3,EP4} have developed a group-theoretical framework for cumulants in non-commutative probability. This approach is based on the identification of a combinatorial word Hopf algebra, $H$, that is defined as an extension of a given non-commutative probability space $(\A,\varphi)$. The coproduct splits into two half-coproducts which provides $H$ with the structure of {unshuffle bialgebra} \cite{Foi}. This induces a splitting of the convolution product in the graded dual $H^*$ into a sum of two products, denoted by $\prec$ and $\succ$, making $(H^*,\prec,\succ)$ a non-commutative shuffle algebra (or dendriform algebra), 
where the convolution product is written as $f \star g = f \succ g + f \prec g$. The half-shuffles together with the convolution product define three exponential maps, $\exp^\star$, $\mathcal{E}_\prec$, $\E_\succ$. Each one of these maps establishes a bijection between the Lie algebra, $\mathfrak{g} \subset H^*$, of infinitesimal characters and the group, $G \subset H^*$, of characters on $H$: if $G \ni \Phi: H \to \mathbb{C}$ is a character, then there exist unique infinitesimal characters $\kappa, \beta, \rho \in \mathfrak{g}$ mapping $H$ to $\mathbb{C}$ such that
\begin{equation}
\label{1}
	\Phi = \exp^\star(\rho) = \mathcal{E}_\prec(\kappa) = \E_\succ(\beta).
\end{equation}
These equations encode the combinatorial description of the monotone, free and Boolean independences, respectively, given by the moment-cumulant formulas \cite{EP1,EP2}. More precisely, the corresponding multivariate moment-cumulant formulas can be obtained by evaluating \eqref{1} in a word $w = a_1a_2 \cdots a_n \in H$:
\allowdisplaybreaks
\begin{align*}
	m_n(a_1,a_2, \ldots, a_n) 
	&:= \Phi(a_1a_2 \cdots a_n)\\	
	&= \mathcal{E}_\prec(\kappa)(w) 
		= \sum_{\pi \in \NC(n)} r_{\pi}(a_1,\ldots,a_n)\\
	&=  \exp^\star(\rho)(w) 
		= \sum_{\pi \in \NC(n)}\frac{1}{\tau(\pi)!} h_\pi(a_1,\ldots,a_n)\\
	&= \E_\succ(\beta)(w) 
		=\sum_{I \in \II(n)}  b_{\pi}(a_1,\ldots,a_n),
\end{align*}
where the free, monotone and Boolean cumulants are given as the images of the corresponding infinitesimal Hopf algebra characters, that is, $\kappa(w)=r_{n}(a_1,a_2, \ldots, a_n)$, $\rho(w)=h_{n}(a_1,a_2, \ldots, a_n)$ and $\beta(w) =  b_{n}(a_1,a_2, \ldots, a_n)$. 
Here, $\NC(n)$ and $\mathcal{I}(n)$ denote the lattices of non-crossing respectively interval set partitions of order $n$ and $b_{\pi}$, $h_\pi$ and $r_{\pi}$ are defined multiplicatively with respect to the blocks in the set partition $\pi$ (see Notation \ref{notat:multiplicative} below).

Explicit relations among Boolean, free and monotone cumulants where studied in detail by Arizmendi et al \cite{AHLV}. These relations have proven to be important in the study of non-commutative probability theory. For instance, Boolean cumulants were successfully used to study free infinite divisibility via the Boolean Bercovici--Pata bijection (see \cite{BN1}, \cite{BN2}, \cite{BN3}). 
From the shuffle algebra perspective, the authors in \cite{EP2} described these relations between cumulants through relations between the three exponential maps and their corresponding logarithms, using the shuffle adjoint operation in the (pre-)Lie algebra of infinitesimal characters.

\smallskip

Variations and extensions of free probability have arisen from both theoretical and applied problems. In this paper we are specifically concerned with the notion of \textit{infinitesimal free probability}, which was introduced in \cite{BS} and \cite{BGN}. This framework consists of a triple $(\A,\varphi,\varphi^\prime)$ where $\AA$ is an algebra and $\varphi$, $\varphi^\prime$ are functionals. The term infinitesimal refers to the intuitive idea that we can see $\varphi^\prime$ as the derivative of a continuous family of distributions converging to $\varphi$. This theory proved to be useful in studying certain random matrices, specifically the asymptotics of finite-rank perturbations \cite{Shl} as well as random matrix models with discrete spectrum \cite{CHS}. Regarding the combinatorial perspective, F\'evrier and Nica \cite{FN} introduced the notion of infinitesimal free cumulants. The vanishing of mixed infinitesimal free cumulants characterises infinitesimal freeness. The close analogy between free cumulants and free infinitesimal cumulants becomes more transparent when combining $\varphi$ and $\varphi^\prime$ into a $\cc$-linear functional $\tilde\varphi$ that takes values in the algebra $\mathbb{G}$ of Grassmann numbers rather than the complex numbers. {\color{black} The infinitesimal versions of Boolean and monotone independence as well as the corresponding notions of cumulants were introduced by Hasebe in the 2011 work \cite{Has}\footnote{We would like to thank T.~Hasebe for bringing this reference to our attention.}.}

{\color{black}  In this paper we first revisit infinitesimal cumulants. To define this notion, }
we start from an infinitesimal non-commutative probability space $(\AA,\varphi,\varphi^\prime)$, and compute {\color{black}  free $\{r_n \colon \A^n\to \mathbb{C} \}_{n \ge 1}$,} Boolean $\{b_n \colon \A^n\to \mathbb{C} \}_{n \ge 1}$ and monotone $\{h_n \colon \A^n\to \mathbb{C} \}_{n \ge 1}$ functional cumulants together with the corresponding moment-cumulant formulas. The idea then is to formally differentiate these formulas to obtain the following relations:
\begin{eqnarray}
 {\color{black} \varphi^\prime_n(a_1, \ldots, a_n)} &=& \sum_{\pi \in \NN\CC(n)} \partial r_\pi(a_1, \ldots, a_n),\\
	\varphi^\prime_n(a_1, \ldots, a_n) &=& \sum_{\pi \in \II(n)} \partial b_\pi(a_1, \ldots, a_n),\\
	\varphi^\prime_n(a_1, \ldots, a_n) &=& \sum_{\pi\in \NN\CC(n)} \frac{1}{\tau(\pi)!} \partial h_\pi(a_1, \ldots, a_n),
\end{eqnarray}
where {\color{black} $\partial  r_\pi(a_1, \dots, a_n)$,} $\partial  b_\pi(a_1, \dots, a_n)$ and $\partial h_\pi(a_1, \dots, a_n)$ are formal derivations (the precise definition is given in Notation \ref{partial.notation} below). These formulas give recursive definitions of infinitesimal {\color{black}free}, Boolean and monotone cumulants. With the use of $\cc$-linear functionals into $\mathbb{G}$ and Möbius inversion formula, explicit formulas {\color{black}can be obtained} for infinitesimal {\color{black} free and} Boolean cumulants in terms of moments. We complement this approach by studying infinitesimal cumulants from the shuffle algebra viewpoint. In this framework, we consider Hopf algebra characters and infinitesimal\footnote{The reader may have noticed an apparent conflict of terminology regarding the use of the term infinitesimal. However, the precise meaning will be clear from context.} characters with values in the commutative algebra $\mathbb{G}$. Previous results from \cite{EP1,EP2} still hold in this framework and the infinitesimal versions of the corresponding formulas are obtained in purely algebraic terms.

The first result in this note shows that analogue relations among cumulants still hold in the infinitesimal framework. We refer the reader to the Appendix (Defs.~\ref{def:partitiontypes} and \ref{defi.irreducible.partitions}) for the definition of non-crossing (irreducible) set partitions.

\begin{theorem}
\label{Thm1}
Let $\NN\CC_{irr}(n)$ be the set of non-crossing irreducible set partitions of order $n$.
The following relations between infinitesimal cumulants $b^\prime,r^\prime,h^\prime$ are satisfied:
\allowdisplaybreaks
\begin{eqnarray}
	b^\prime_n (a_1, \dots, a_n) 
		&=& \sum_{\pi\in \NN\CC_{irr}(n)} \partial r_\pi(a_1, \dots, a_n)  \label{InfBooFree}\\
	r^\prime_n (a_1, \dots, a_n) 
		&=& \sum_{\pi\in \NN\CC_{irr}(n)} (-1)^{|\pi|-1} \partial b_\pi(a_1, \dots, a_n) \label{InfFreeBoo}\\
	b^\prime_n (a_1, \dots, a_n) 
		&=& \sum_{\pi\in \NN\CC_{irr}(n)} \frac{1}{\tau(\pi)!}\partial h_\pi(a_1, \dots, a_n) \label{InfBooMon} \\
	r^\prime_n (a_1, \dots, a_n) 
		&=& \sum_{\pi\in \NN\CC_{irr}(n)} \frac{(-1)^{|\pi|-1}}{\tau(\pi)!} \partial h_\pi(a_1, \dots, a_n). \label{InfFreeMon}
\end{eqnarray}
\end{theorem}

We show the statements of this theorem using both an intuitive and the shuffle algebra approach. One of the main aims of this paper is to understand the usefulness of the latter in studying combinatorial questions in non-commutative probability using Hopf algebraic tools.

As an application, we propose an infinitesimal analog of the Boolean Bercovici--Pata map at the algebraic level. Recall that this map assigns to each distribution $\mu$ another distribution $\bb(\mu)$, such that $r_n(\bb(\mu))=b_n(\mu)$. On the other hand, we can define an infinitesimal version of the main transform studied in \cite{BN2} and prove that analogue results still hold. This is summarized in the following theorem (the precise definitions and notations are given in Section \ref{sec:Bercovici.Pata}).

\begin{theorem}
	\label{Thm1.2}
Given an infinitesimal law $\tilde\mu$ and a real $t \geq 0$, we can define $\tilde\bb_t$ to be the function sending infinitesimal laws to infinitesimal laws, such that
$$
	\tilde\bb_t(\tilde\mu)=\pai \tilde\mu^{\boxplus 1+t}\pad^{\uplus \frac{1}{1+t}}.
$$
Then we have the following:

\begin{enumerate}
\item The maps $\{\tilde\bb_t|t\geq 0\}$ satisfy that 
$$
	\tilde\bb_s\circ\tilde\bb_t=\tilde\bb_{s+t} \qquad \forall s,t \geq 0.
$$ 
\item For all $n$, $\tilde r_n (\tilde\bb_1(\tilde\mu))= \tilde b_n(\tilde\mu)$.
This means that $\tilde\bb_1(\tilde\mu)=\tilde\bb(\tilde\mu)$ is the law obtained from the infinitesimal Boolean Bercovici--Pata map.
\end{enumerate}
\end{theorem}

\smallskip

In addition to the foregoing introductory section, the rest of the paper is divided into five more sections. In Section \ref{sec:preliminaries} we review Boolean, free and monotone cumulants, as well as the Hopf algebraic approach to the corresponding moment-cumulant relations.  In Section \ref{sec:InfCumulants} we first review the notion of infinitesimal non-commutative probability space together with { infinitesimal cumulants. We also study the infinitesimal Boolean cumulant multivariable series.}  In Section \ref{sec:ShuffleInfCumulants} we show how the shuffle algebra approach naturally extends to the infinitesimal setting.  Section \ref{sec:RelationsInfCum} is devoted to the proof of Theorem \ref{Thm1}. As an application we prove Theorem \ref{Thm1.2} in Section \ref{sec:Bercovici.Pata}. In an Appendix (Section \ref{sec:appendix}) we collect some basics on set partitions and the different types of independences.

\medskip

Acknowledgements: The authors would like to thank Alexandru Nica for his valuable comments. We also would like to thank Takahiro Hasebe for bringing reference \cite{Has} to our attention. This work was partially supported by the project Pure Mathematics in Norway, funded by Trond Mohn Foundation and Troms{\o} Research Foundation. Part of this research was conducted while the third author was visiting NTNU's Department of Mathematical Sciences in Trondheim. He expresses his gratitude for the warm hospitality and stimulating atmosphere at NTNU, and for travel support from NSERC, Canada. The third author also receives support from CONACyT (Mexico) via the scholarship 714236. 


\section{Preliminaries}
\label{sec:preliminaries}

When it is not specified, all the objects (vector spaces, algebras, coalgebras, pre-Lie algebras, etc.) will be taken over the field of complex numbers, denoted $\mathbb{C}$.


\subsection{Cumulants in non-commutative probability}
\label{ssec:cumulants}

A \textit{non-commutative probability space} (ncps) is a pair $(\A,\varphi)$, where $\A$ is a unital algebra over $\mathbb{C}$ and $\varphi \colon \A \to \mathbb{C}$ is a linear functional, such that $\varphi(1_\A) =1$. The \textit{$n$-th multivariate moment} is the multilinear functional $\varphi_n: \A^n \to \mathbb{C}$, such that $\varphi_n(a_1, \ldots, a_n):=\varphi(a_1 \cdot_{\!\scriptscriptstyle{\A}} \cdots \cdot_{\!\scriptscriptstyle{\A}} a_n) \in \mathbb{C}$, for elements $a_1, \ldots, a_n \in \A$, where $ \cdot_{\!\scriptscriptstyle{\A}}$ stands for the product in $\A$. 

In this framework we can define the notions of free, Boolean and monotone independence (see Appendix \ref{ssec:independences}). The paper at hand is concerned with cumulants, which are the major combinatorial tool to handle these three types of independence. Let us first fix some notation.

\begin{notation}\label{notat:multiplicative}
Let $\pi$ be an element in $\PP(n)$, the set partitions of $[n]:=\{1,\dots,n\}$ (see Appendix). Given any family of multilinear functionals $\{f_m \colon \A^{m}\to \mathbb{C}\}_{m\geq1}$, we denote 
$$
	f_\pi(a_1, \ldots, a_n):=\prod_{V \in \pi}f_{|V|}(a_V).
$$ 
Here, $V:=\{{v_1}, \ldots, {v_k}\} \in \pi$ is a block of $\pi$ (where ${v_1} < \cdots < {v_k}$ are in natural order) and we define $f_{|V|}(a_V):=f_{k}(a_{v_1},\ldots,a_{v_k})$.
\end{notation}

Now we are ready to go over the definitions of free, Boolean and monotone cumulants.

\begin{definition} (Cumulants)
Let $(\A,\varphi)$ be a ncps.
\begin{itemize}
\item \textit{Free cumulants} form the family of multilinear functionals $\{r_n \colon \A^{n}\to \mathbb{C}\}_{n\geq1}$ recursively defined by the following formula:
\begin{equation}
\label{FreeMomCum}
	\varphi_n(a_1, \dots, a_n) =\sum_{\pi \in \NN\CC(n)} r_\pi(a_1, \dots, a_n).
\end{equation} 

\item \textit{Boolean cumulants} form the family of multilinear functionals $\{ b_n :\A^{n}\to \mathbb{C}\}_{n \ge 1}$ recursively defined by the following formula:
\begin{equation}
\label{BooMomCum}
	\varphi_n(a_1, \dots, a_n) =\sum_{\pi \in \II(n)}  b_\pi(a_1, \dots, a_n).
\end{equation} 

\item \textit{Monotone cumulants} form the family of multilinear functionals $\{h_n \colon \A^n\to \mathbb{C} \}_{n \ge 1}$ recursively defined by the following formula:
\begin{equation}
\label{MonMomCum}
	\varphi_n(a_1, \dots, a_n)=\sum_{\pi\in \NN\CC(n)} \frac{1}{\tau(\pi)!} h_\pi(a_1, \dots, a_n).
\end{equation} 
\end{itemize}
\end{definition}

\begin{remark}
Cumulants are well defined since the right-hand sides of the foregoing equations contain only one term of size $n$ ($r_n,b_n$ or $h_n$) while the other terms are monomials of cumulants of smaller sizes (though, the total degree $n$ is preserved). This defines a triangular system of equations to compute cumulants to all orders.

We can invert the above equations giving free and Boolean cumulants directly by the formulas
\allowdisplaybreaks
\begin{align*}
	r_n(a_1, \dots, a_n)
		&=\sum_{\sigma \in \NN\CC(n)} \text{M\"ob}(\sigma, 1_n) \varphi_\sigma(a_1, \dots, a_n),\\
	b_n(a_1, \dots, a_n) 
		&=\sum_{\sigma \in \II(n)} (-1)^{|\sigma|-1}\varphi_\sigma(a_1, \dots, a_n).
\end{align*}
The reason why these formulas are equivalent to the previous definition is due to Möbius inversion in the lattices of non-crossing and interval partitions, respectively.
\end{remark}

Another way in which we can relate moments with cumulants is via formal power series. Let $\cc\langle z_1,\dots, z_k\rangle$ denote the algebra of polynomials in non-commuting indeterminates $z_1,\dots, z_k$, and let $\cc\langle\langle z_1,\dots, z_k\rangle\rangle$ be the set of power series with complex coefficients in the non-commuting indeterminates $z_1,\dots, z_k$. {Now, fix a $k$-tuple of elements $a=(a_1,\dots, a_k)$ from $\A$ and recall that $\varphi(1_\A)=1$. Then we can construct the \textit{multivariate moment series} $M_a$ of $a$}:
$$
	M_a(z_1, \dots, z_k)
	=\sum_{n=1}^\infty \sum_{i_1, \dots, i_n=1}^k 
	\varphi_n(a_{i_1},\dots, a_{i_n}) z_{i_1} \cdots z_{i_n} \in \cc\langle\langle z_1,\dots, z_k\rangle\rangle.
$$

{Analogously, we can construct the multivariate series, $R_a$, $B_a$, $H_a$ in $\cc\langle\langle z_1,\dots, z_k\rangle\rangle$, that have free, Boolean and monotone cumulants, respectively, as coefficients, such that}
\begin{align*}
	R_a(z_1, \dots, z_k)
	&=\sum_{n=1}^\infty \sum_{i_1, \dots, i_n=1}^k r_n(a_{i_1},\dots, a_{i_n}) z_{i_1} \cdots z_{i_n},\\
	B_a(z_1, \dots, z_k)
	&=\sum_{n=1}^\infty \sum_{i_1, \dots, i_n=1}^k b_n(a_{i_1},\dots, a_{i_n}) z_{i_1} \cdots z_{i_n},\\
	H_a(z_1, \dots, z_k)
	&=\sum_{n=1}^\infty \sum_{i_1, \dots, i_n=1}^k h_n(a_{i_1},\dots, a_{i_n}) z_{i_1} \cdots z_{i_n}.
\end{align*}

It is known (see \cite{NS}) that the moment series $M_a$ and the \textit{R-transform} $R_a$ satisfy the relation
\begin{equation}
\label{eq.mom.cum.gen.series.rel}
	M_a(z_1, \dots, z_k) 
	=R_a \Big(z_1(1+M_a(z_1, \dots, z_k)), \dots, z_k(1+M_a(z_1, \dots, z_k))\Big).
\end{equation}
Also, it can be seen that $M_a$ and the \textit{$\eta$-series} $B_a$ (typically named $\eta_a$ instead of $B_a$) satisfy 
\begin{equation}
\label{boolean.gen.series.rel}
	M_a(z_1, \dots, z_k) =(1+M_a(z_1, \dots, z_k))B_a(z_1, \dots, z_k).
\end{equation}
We are not aware of a direct relation between $M_a$ and $H_a$, and we refer the reader to Section 6 of \cite{HS2} for a discussion on this.

As mentioned before, explicit relations among Boolean, free and monotone cumulants have been useful in the study of non-commutative probability theory. In this paper we will focus on the following formulas that can be found in \cite{AHLV}:
\begin{eqnarray}
	b_n (a_1, \dots, a_n) 
		&=& \sum_{\pi\in \NN\CC_{irr}(n)} r_\pi(a_1, \dots, a_n)  \label{BooFree}\\
	r_n (a_1, \dots, a_n) 
		&=& \sum_{\pi\in \NN\CC_{irr}(n)} (-1)^{|\pi|-1}  b_\pi(a_1, \dots, a_n) \label{FreeBoo} \\
	b_n (a_1, \dots, a_n) 
		&=& \sum_{\pi\in \NN\CC_{irr}(n)} \frac{1}{\tau(\pi)!} h_\pi(a_1, \dots, a_n) \label{BooMon} \\
	r_n (a_1, \dots, a_n) 
		&=&  \sum_{\pi\in \NN\CC_{irr}(n)} \frac{(-1)^{|\pi|-1}}{\tau(\pi)!} h_\pi(a_1, \dots, a_n) \label{FreeMon}
\end{eqnarray}


{{We will also remark on the inverted multivariate relations of \eqref{BooMon} and \eqref{FreeMon}, expressing monotone in terms of Boolean and free cumulants, respectively, which were given in the recent work \cite{Magnus2021}.}}


\subsection{Hopf--algebraic approach to non-commutative independence}
\label{ssec:HopfPreliminaries}

Let $(\A,\varphi)$ be a ncps. We consider the framework of shuffle algebra for non-commutative probability. Consider the tensor algebra $T_+(\A) = \bigoplus_{n>0} \A^{\otimes n}$ and define the double tensor algebra
$$
	H := T(T_+(\A)) = {{\mathbb{C}\mathbf{1} \oplus \bigoplus_{n > 0} T_+(\A)^{\otimes n}.}}
$$ 
This is a graded connected non-commutative non-cocommutative Hopf algebra with unit $\mathbf{1}$ {{and counit $\varepsilon: H \to \mathbb{C}$. The latter maps $H_+:=\bigoplus_{n > 0} T_+(\A)^{\otimes n}$ to zero and the unit $\mathbf{1}$ to one, i.e. $\varepsilon(z\mathbf{1})=z \in \mathbb{C}$. Product and coproduct are given as follows.}}
\\\textit{Product}: Elements in $T_+(\A)$ are written as words, i.e., $w=a_1\cdots a_k := a_1\otimes \cdots \otimes a_k$. The length of a word, that is, its number of letters is denoted $\mathrm{deg}(w)=k$. The product in $H$ is denoted using the bar-notation, that is, given words $w_1,w_2 \in T_+(\A)$, then their product is $w_1 | w_2 \in H$. Hence, elements in $H$ are denoted by $w_1 | \cdots | w_n$, for words $w_1,\ldots,w_n \in T_+(\A)$. 
\\\textit{Coproduct}: We start with a canonically ordered subset $S=\{s_1 < \cdots < s_m\} \subset [n]:=\{1,\ldots,n\}$ and define for $a_1\cdots a_n \in \A^{\otimes n}$ the word $a_S = a_{s_1}a_{s_2}\cdots a_{s_m}$. We denote by $J_1,\ldots,J_k$ the interval components of $[n]\backslash S$. Then, we define the linear map $\Delta: T_+(\A)  \to H \otimes H$
\begin{equation}
\label{coproduct}
	a_1\cdots a_n  \mapsto  \displaystyle \Delta(a_1\cdots a_n )
		:=\sum_{S\subseteq [n]} a_S \otimes a_{J_1}|\cdots | a_{J_k}
\end{equation}
and extend it multiplicatively to a map $\Delta$ on all of $H$, where $\Delta(\mathbf{1}) := \mathbf{1} \otimes \mathbf{1}$. This coproduct $\Delta$ turns $H$ into a graded connected non-commutative non-cocommutative Hopf algebra \cite{EP1}. 
\par The coproduct \eqref{coproduct} can be split into so-called left and right half-shuffle coproducts, $\Delta_\prec$ respectively $\Delta_\succ$, such that 
$$
	\Delta = \Delta_\prec+\Delta_\succ.
$$ 
More precisely,
\begin{eqnarray*}
	\Delta_\prec(a_1\cdots a_n) = \sum_{1\in S\subseteq [n]} a_S \otimes a_{J_1}|\cdots | a_{J_k},\\ 
	\Delta_\succ(a_1\cdots a_n) = \sum_{1\not\in S\subset [n]} a_S \otimes a_{J_1}|\cdots | a_{J_k},
\end{eqnarray*}
where $a_1\cdots a_n\in \A^{\otimes n}$. These maps are extended to $H$ by defining for $w_1 \in T_+(\A)$ and $w_2 \in H$
$$
	\Delta_\prec(w_1| w_2):=\Delta_\prec(w_1)\Delta(w_2),
$$
and similarly for $\Delta_\succ$. It turns out that $H$ together with $\Delta_\prec$, $\Delta_\succ$ becomes a \textit{unital unshuffle bialgebra} \cite{EP2,EP4}, a dual notion of non-commutative shuffle algebra \cite{Foi}.

{{\begin{notation}\label{notat:dual}
We denote by $H^\ast := \operatorname{Hom}(H,\mathbb{C})$ the graded dual space of $\mathbb{C}$-linear maps from $H$ to $\mathbb{C}$. It becomes a unital associative algebra when equipped with the convolution product defined in terms of the coproduct \eqref{coproduct}
\begin{equation}
\label{eq:convol}
	f\star g = m_\mathbb{C}(f\otimes g)\Delta,
\end{equation}
for $f,g\in \operatorname{Hom}(H,\mathbb{C})$, where $m_\mathbb{C}$ stands for the multiplication in $\mathbb{C}$. The unit is given by the Hopf algebra counit $\varepsilon: H\to\mathbb{C}$. We say that $\Phi \in \operatorname{Hom}(H,\mathbb{C})$ is a Hopf algebra \textit{character} if $\Phi(\mathbf{1})=1$ and $\Phi$ is multiplicative with respect to the product in $H$, i.e., $\Phi(w_1 | w_2 )=\Phi(w_1) \Phi(w_2)$. In the same way, we say that $\kappa \in \operatorname{Hom}(H,\mathbb{C})$ is an \textit{infinitesimal character} if $\kappa(\mathbf{1})=0$ and $\kappa(w_1|w_2)=0$ for any $w_1,w_2 \in H_+.$  
\end{notation}
One can show that the set $G$ of characters forms a group with respect to the convolution product \eqref{eq:convol}. The inverse of an element in $G$ is given by composition with the antipode of $H$. Moreover, the set $\mathfrak{g}$ of infinitesimal characters forms a Lie algebra. From the classical theory of Hopf algebras \cite{Sweedler69}, we know that the exponential map with respect to the convolution product provides a set bijection between $G$ and $\mathfrak{g}$.
}}

\par If we replace the coproduct $\Delta$ in \eqref{eq:convol} by the half-shuffle coproducts, $\Delta_\prec$ and $\Delta_\succ$, we obtain two new linear maps 
\begin{eqnarray*}
	f \prec g &=& m_\mathbb{C}(f\otimes g)\Delta_\prec, \\
	f \succ g &=& m_\mathbb{C}(f\otimes g)\Delta_\succ.
\end{eqnarray*}
Neither of these two operations is  associative. Indeed, they satisfy the so-called {\it{shuffle identities}}:
\begin{eqnarray}
	(f \prec g)\prec h &=& f \prec(g \star h)        		\label{A1}\\
  	(f \succ g)\prec h &=& f \succ(g\prec h)   			\label{A2}\\
        f \succ(g\succ h)  &=& (f \star g)\succ h,	      	        \label{A3}
\end{eqnarray}
such that the following theorem can be proven.

\begin{theorem}[Proposition 6 in \cite{EP1}]
$(\operatorname{Hom}(H,\mathbb{C}), \prec,\succ)$ is a unital shuffle algebra.
\end{theorem}

The three products, $(\star, \prec,\succ)$, defined on $\operatorname{Hom}(H,\mathbb{C})$, imply an intricate structure on the dual space of $H$. In general, given a unital shuffle algebra $(D,\prec,\succ)$ with associative product $a \star b = a\prec b+a\succ b$, for $a,b\in D$, we can define the usual exponential map relative to the product $\star$ by
$$
	\exp^\star (a) := 1_D + \sum_{n\geq 1} \frac{a^{\star n}}{n!}.
$$
In an analogous way, we define the half-shuffle exponential maps, $\E_\prec$ and $\E_\succ$, by
$$
	\E_\prec(a) := 1_D + \sum_{n\geq 1} a^{\prec n} ,
	\qquad     
	\E_\succ(a) := 1_D + \sum_{n\geq 1} a^{\succ n},
$$
where $a^{\prec n}:=a \prec(a^{\prec n-1})$, $a^{\prec 0}=1_D $, and analogously for $a^{\succ n}$. These new exponential-type maps have inverses with respect to the associative product $\star$, given by
\begin{eqnarray}
\label{InvPrec}
	\E_\prec^{-1}(a)& =& \E_\succ(-a),\\
\label{InvSucc}
	\E_\succ^{-1}(a) &=& \E_\prec(-a).
\end{eqnarray}

We return to the concrete example of shuffle algebra provided by $(\operatorname{Hom}(H,\mathbb{C}), \prec,\succ)$. First, we notice that the three exponential-type series, $\exp^\star$, $\E_\prec$ and $\E_\succ$, are indeed finite sums when evaluated on a word of finite length. It was shown in \cite{EP1} that an interesting connection among these maps is provided by the fact that the set $G$ of characters is bijectively related with the set $\mathfrak{g}$ of infinitesimal characters via these three exponential maps. More precisely,

\begin{theorem}[\cite{EP1}]
For $\Phi$ a character, there exist a unique triple $(\kappa,\beta,\rho)$ of infinitesimal characters such that 
\begin{equation}
\label{key}
	\Phi = \exp^\star(\rho) = \E_\prec(\kappa) = \E_\succ(\beta).
\end{equation}
In particular, we have that $\kappa$ and $\beta$ are the unique solutions of the half-shuffle fixed point equations 
\begin{equation}
\label{leftfpointeq}
	\Phi = \varepsilon + \kappa \prec\Phi
\end{equation}
respectively 
\begin{equation}
\label{rightfpointeq}
	\Phi = \varepsilon + \Phi \succ \beta.
\end{equation} 
Conversely, given $\alpha \in \mathfrak{g}$, then $\exp^\star(\alpha)$, $\mathcal{E}_\prec(\alpha)$ and $\mathcal{E}_\succ(\alpha)$ are characters.
\end{theorem}

\par The link with non-commutative probability appears when we lift the linear functional $\varphi$ to the character $\Phi \in G$, such that the $n$-th multivariate moment $\Phi(w):=\varphi_n(a_1, \dots, a_n)$, for $w=a_1\cdots a_n\in T_+(\A)$. It turns out that, for the triple of infinitesimal characters given by the above theorem, $\kappa \in \mathfrak{g}$ evaluated in the word $w = a_1\cdots a_n \in T_+(\mathcal{A})$ can be identified with the free cumulant in the algebra elements $a_1,\ldots,a_n$. In the same way, $\beta \in \mathfrak{g}$ corresponds to Boolean cumulants and $\rho \in \mathfrak{g}$ corresponds to monotone cumulants. More specifically, we have that $\kappa(w)=r_{n}(a_1,a_2, \ldots, a_n)$, $\beta(w) =  b_{n}(a_1,a_2, \ldots, a_n)$ and $\rho(w)=h_{n}(a_1,a_2, \ldots, a_n)$. In particular, these evaluations allow us to obtain the corresponding free, Boolean and monotone moment-cumulant formulas via shuffle algebra
\begin{align}
	\E_\prec(\kappa)(a_1 \cdots a_n)
	&= \sum_{\pi\in \NC(n)} r_\pi(a_1,\ldots,a_n)\\
	\E_\succ(\beta)(a_1 \cdots a_n)
	&= \sum_{\pi\in \mathcal{I}(n)} b_\pi(a_1,\ldots,a_n)\\
	 \exp^\star(\rho)(a_1 \cdots a_n)
	&= \sum_{\pi\in \mathcal{M}(n)} \frac{1}{|\pi|!}h_\pi(a_1,\ldots,a_n).
\end{align}


\medskip

There is a natural action of the group $G$ on its Lie algebra $\mathfrak{g}$, i.e., for $\Psi \in G$ and $\alpha \in \mathfrak{g}$ the adjoint action 
\begin{equation}
\label{adjoint}
	\mathrm{Ad}_\Psi(\alpha):=\Psi^{-1} \star \alpha \star \Psi. 
\end{equation}
Moreover, the shuffle structure permits to define another action, i.e., the shuffle adjoint action
\begin{equation}
\label{shuffleadjoint}
	\theta_\Psi(\alpha):=\Psi^{-1} \succ \alpha \prec \Psi. 
\end{equation}
One verifies that for $\Psi \in G$ and $\alpha \in \mathfrak{g}$, we have that $\theta_\Psi(\alpha) \in \mathfrak{g}$. The shuffle axioms \eqref{A1}-\eqref{A3} yield
$$
	\theta_\Psi \circ \theta_\Phi(\alpha)
	=\Psi^{-1} \succ( \Phi^{-1} \succ \alpha \prec \Phi )\prec \Psi
	= \theta_{\Phi \star \Psi} (\alpha).	
$$
The fixed point equations \eqref{leftfpointeq} and \eqref{rightfpointeq}  imply that $\kappa \prec \Phi = \Phi \succ \beta$ which is equivalent to the fundamental relation between free and Boolean cumulants
\begin{equation}
\label{FreeBoolShuffle}
	\beta = \Phi^{-1} \succ \kappa \prec \Phi = \theta_{\Phi}(\kappa).
\end{equation}
The combinatorial expression of the latter is given by the following result.

\begin{lemma}[\cite{EP4}]
\label{lem:CombAdjoint}
Let $\Phi$ be a character in $G$ and $\alpha,\kappa,\beta$ be infinitesimal characters in $\mathfrak{g}$ such that $\Phi = \E_\prec(\kappa) = \E_\succ(\beta)$ (i.e., $\kappa$ and $\beta$ correspond to the free and Boolean cumulant infinitesimal characters associated to $\Phi$). For a word $w = a_1\cdots a_n$ we have that
\begin{eqnarray}
	\theta_\Phi(\alpha)(w) 
	&=& \sum_{1,n\in S\subseteq [n]} \alpha(a_S) \Phi(a_{J_{[n]}^S})\\ 
\label{CombTheta}
	&=& \sum_{\pi\in \NC_{irr}(n)} \alpha_{|V_1|}(a_{V_1}) 
	\prod_{\substack{W \in \pi\\W\neq V_1}} r_{|W|}(a_W),
\end{eqnarray}
and
\begin{eqnarray}
	\theta_{\Phi^{-1}}(\alpha)(w) 
	&=& \sum_{1,n\in S\subseteq [n]} \alpha(a_S) \Phi^{-1}(a_{J_{[n]}^S})\\ 
\label{CombThetaInv}
	&=& \sum_{\pi\in \NC_{irr}(n)} (-1)^{|\pi|-1}\alpha_{|V_1|}(a_{V_1}) 
	\prod_{\substack{W\in \pi \\W\neq V_1}} b_{|W|}(a_W),
\end{eqnarray}
where $V_1\in \pi\in \NC_{irr}(n)$ denotes the unique outer block with $1,n \in V_1 \subset [n]$.
\end{lemma}

In particular, if we take $\alpha = \kappa$ in Equation \eqref{CombTheta}, we obtain the relation between free and Boolean cumulants \eqref{BooFree}. On the other hand, taking $\alpha = \beta$ in Equation \eqref{CombThetaInv}, we obtain the converse relation \eqref{FreeBoo} expressing Boolean cumulants in terms of free cumulants.
 

Based on the shuffle adjoint action, we can define the following right action of $G$ on itself. Let $\Psi_i=\mathcal{E}_\prec(\kappa_i) \in G$, $\kappa_i \in \mathfrak{g}$, $i=1,2,3$
\begin{equation}
\label{shuffleadjointgroup}
	\Psi_1\ \bole \Psi_2:=\mathcal{E}_\prec(\theta_{\Psi_2}(\kappa_1)).
\end{equation}
Observe that the shuffle axioms \eqref{A1}-\eqref{A3} imply that 
\begin{align*}
	(\Psi_1\ \bole \Psi_2)\ \bole \Psi_3 
	&= \mathcal{E}_\prec(\theta_{\Psi_3}\circ\theta_{\Psi_2}(\kappa_1))\\
	&=\mathcal{E}_\prec(\theta_{\Psi_2 \star \Psi_3}(\kappa_1))\\
	&=\Psi_1\ \bole (\Psi_2 \star \Psi_3). 
\end{align*}
It can be shown that the $\bole$-action is the subordination operation, defined by Lenczewski in \cite{Len} to describe the decomposition of the free additive convolution. Indeed, we have that free additive convolution of $\Psi_i=\mathcal{E}_\prec(\kappa_i) \in G$, $\kappa_i \in \mathfrak{g}$, $i=1,2$, is given by 
\begin{equation}
\label{addconv}
	\Psi_1 \boxplus \Psi_2 =\mathcal{E}_\prec(\kappa_1 + \kappa_2) = \Psi_1 \star (\Psi_2\ \bole \Psi_1).
\end{equation}
Similarly for the Boolean case, i.e., for $\Psi_i=\mathcal{E}_\succ(\beta_i) \in G$, $\beta_i \in \mathfrak{g}$, $i=1,2$, we can define a left action of $G$ on itself
\begin{equation}
\label{leftshuffleadjointgroupaction}
	\Psi_1\ \bori \Psi_2 
	:= \mathcal{E}_\succ(\theta_{\Psi_1^{-1}}(\beta_2)),
\end{equation}
such that 
\begin{equation}
\label{addbooleanconv}
	\Psi_1 \uplus \Psi_2 
	= \mathcal{E}_\succ(\beta_1 + \beta_2) 
	= (\Psi_2\ \bori \Psi_1) \star \Psi_2.
\end{equation}

Note that the second equality in \eqref{addconv} (and also in \eqref{addbooleanconv}) is defined using \eqref{key}. It involves a rather non-trivial relation between the shuffle adjoint action and the Baker--Campbell--Hausdorff formula \cite{EP4}. In fact, identity \eqref{key} permits to write $\mathcal{E}_\prec(\kappa)=\exp^\star(\Omega'(\kappa))$ and $\mathcal{E}_\succ(\beta)=\exp^\star(-\Omega'(-\beta))$, where $\Omega' \colon \mathfrak{g} \to \mathfrak{g}$ is the \textit{pre-Lie Magnus expansion}, a highly non-linear transformation on the Lie algebra $\mathfrak{g}$ of infinitesimal characters 
$$
	\Omega'(\kappa):=\sum_{n \ge 0}\frac{B_n}{n!} L^{(n)}_{\Omega'(\kappa) \rhd}(\kappa)=\kappa -\frac{1}{2} \kappa \rhd \kappa + 
	\frac{1}{4} (\kappa \rhd \kappa)\rhd \kappa + \frac{1}{12} \kappa \rhd (\kappa \rhd \kappa)  + \cdots.
$$
Here the $B_n$ are the Bernoulli numbers and the product 
$$
	\alpha \rhd \beta := \alpha \succ \beta - \beta \prec \alpha
$$
satisfies the left pre-Lie relation 
\begin{equation}
\label{leftpreLie}
	(\alpha \rhd \beta) \rhd \gamma - \alpha \rhd (\beta \rhd \gamma) 
	= (\beta \rhd \alpha) \rhd \gamma - \beta \rhd (\alpha \rhd \gamma),     
\end{equation}
where $L^{(n)}_{\alpha \rhd} (\beta):=\alpha \rhd (L^{(n-1)}_{\alpha \rhd}(\beta))$, $L^{(0)}_{\alpha \rhd} (\beta)=\beta$. Note that a pre-Lie algebra is Lie admissible, i.e.
$$
	\alpha \rhd \beta - \beta \rhd \alpha =[\alpha,\beta] = \alpha \star \beta - \beta \star \alpha.
$$ 
Hence, the Lie algebra $\mathfrak{g}$ of infinitesimal characters is a pre-Lie algebra. The compositional inverse of $\Omega'$ is given by
$$
	W(\kappa)
	:= \frac{e^{L_{\kappa \rhd}} - 1}{L_{\kappa}}(\kappa)
	= \kappa + \sum_{n>0} \frac{1}{(n+1)!}L^{(n)}_{\kappa \rhd}(\kappa).
$$

The map $\Omega'$ permits to relate Boolean and free cumulants with monotone cumulants, as
\begin{equation}
\label{PreLieMagnus0}
	\exp^\star(\Omega'(\kappa))=\exp^\star(-\Omega'(-\beta))=\exp^\star(\rho).
\end{equation}
See \cite{EP2,EP4} for more details and the recent work \cite{Magnus2021} for explicit formulas deduced from \eqref{PreLieMagnus0}, expressing monotone in terms of Boolean and free cumulants.



\section{Free, Boolean and monotone infinitesimal cumulants}
\label{sec:InfCumulants}
{\color{black}
In this section, we review the notion of infinitesimal free cumulants (following \cite{FN}) and the analogue notions of infinitesimal Boolean and monotone cumulants (\cite{Has})}. To this end, the intuitive idea is to look at the definitions that we already have in non-commutative probability and formally differentiate the formulas to get a natural definition for the infinitesimal cumulants. We will also check how these definitions are understood when we work with the algebra of Grassmann numbers and $\cc$-linear maps, and how they lead to equivalent definitions.


\subsection{Infinitesimal non-commutative probability spaces}
\label{ssec:incps}

We begin by extending our notion of ncps to include another linear functional.

\begin{definition}
\label{def:incps}
An \textit{infinitesimal non-commutative probability space} (incps) is a triple $(\A,\varphi,\varphi^\prime)$, where $(\A,\varphi)$ is a ncps and $\varphi^\prime : \A\to\mathbb{C}$ is a linear functional satisfying $\varphi^\prime(1_\A)=0$.
\end{definition}

The \textit{$n$-th multivariate infinitesimal moment} is defined to be the multilinear functional $\varphi^\prime_n: \A^n \to \mathbb{C}$ satisfying $\varphi^\prime_n(a_1, \ldots, a_n):=\varphi^\prime(a_1 \cdot_{\!\scriptscriptstyle{\A}} \cdots \cdot_{\!\scriptscriptstyle{\A}} a_n) \in \mathbb{C}$, for $a_1, \ldots, a_n \in \A$.  Intuitively, we want to think of these infinitesimal moments, $\varphi^\prime_n$, as being the formal differentiations of the usual moments $\varphi_n$. Thus, formal differentiation of products of moments of the form $\varphi_\pi(a_1,\ldots,a_n)=\prod_{W\in \pi} \varphi_{|W|}(a_W)$ with $\pi\in\PP(n)$ should be the result of applying the Leibniz rule:
$$
	\sum_{V \in \pi} \varphi^\prime_{|V|}(a_V) \prod_{\substack{W \in \pi\\ W \neq V}} \varphi_{|W|}(a_W).
$$
Since this kind of expression will appear constantly, let us fix the notation (we follow \cite{Min}). 

\begin{notation}
\label{partial.notation}
Given $\pi \in \PP(n)$ and a sequence of pairs $(f_n, f^\prime_n)$ of multilinear maps $f_n, f^\prime_n: \AA^{n}\rightarrow \cc$ we denote
$$
	\partial f_\pi(a_1, \dots, a_n):= \sum_{V \in \pi} f^\prime_{|V|}(a_V) \prod_{\substack{W \in \pi\\ W \neq V}} f_{|W|}(a_W).
$$
\end{notation}

A nice way to formalize our previous intuitive considerations is by using the notion of algebra of Grassmann numbers. As mastered in \cite{FN}, this captures the essence of incps in a way that resembles ncps, making it much easier to handle these objects. The algebra of Grassmann numbers $\mathbb{G} = \{z+\hbar w\,:\,z,w\in\mathbb{C}\},$ is defined as a 2-dimensional vector space over $\mathbb{C}$ with commutative multiplication given by 
$$
	( z_1 + \hbar w_1) \cdot ( z_2 + \hbar w_2 ) =z_1 z_2 + \hbar ( z_1 w_2 + w_1 z_2) \qquad \forall z_1,w_1,z_2,w_2\in\cc,
$$
where we formally have that $\hbar^2=0$.

Now, the key idea in \cite{FN} is to keep track of the two functionals, $\varphi$ and $\varphi^\prime$, and merge them into one single $\mathbb{G}$-valued map 
$$
	\tilde{\varphi}:=\varphi+\hbar\varphi^\prime.
$$ 
We can equivalently think of infinitesimal non-commutative probability spaces as pairs $(\AA,\tilde{\varphi})$ where $\AA$ is a unital algebra over $\cc$ and $\tilde{\varphi}: \AA \to \ggg$ is a $\cc$-linear map with $\tilde{\varphi} (1_{\AA}) = 1$. In general, we are going to use this idea of merging whenever we have pairs of functionals, $f_n$, $f^\prime_n$.

\begin{notation}
\label{tilde.notation}
Given two linear functionals $f_n,f^\prime_n:\AA^n\to\cc$, we will denote by $\tilde{f_n}$ the $\cc$-linear map $\tilde{f}_n:\AA^n\to \ggg$ such that 
$$
	\tilde{f}_n(a_1, \dots, a_n):=f_n(a_1, \dots, a_n)+\hbar f^\prime_n(a_1, \dots, a_n).
$$
As before, we will use the convention
\begin{equation}
	\tilde{f}_\pi(a_1, \dots, a_n):= \sideset{}{^\ggg}\prod_{V\in\pi} \tilde{f}_{|V|}(a_V).
\end{equation}
As indicated, the product is in the algebra $\ggg$.
\end{notation}

\begin{remark}
A straightforward computation shows that $\tilde{f}_\pi = f_\pi +\hbar  \partial f_\pi$.
\end{remark}


\subsection{Infinitesimal free cumulants}
\label{ssec:inffree}

With both the notion of incps and the notation in place, we may consider cumulants. The notion of infinitesimal free cumulants was introduced in \cite{FN}. 

\begin{definition}
\cite{FN}\label{def:inffreecumul}
	Let $(\A,\varphi,\varphi^\prime)$ be an incps. The \textit{infinitesimal free cumulants with respect to $(\A,\varphi,\varphi^\prime)$} form the family of multilinear functionals $\{r_n^\prime:\A^{n}\to \mathbb{C}\}_{n\geq1}$ recursively defined by:
\begin{equation}
\label{InfFreeMomCum}
	\varphi_n^\prime(a_1,\ldots,a_n) 
	= \sum_{\pi\in \NC(n)} \partial r_\pi(a_1,\ldots,a_n) 
	\qquad \forall n\geq 1,\ a_1,\ldots,a_n\in \AA.
\end{equation}  
\end{definition}

Notice that this formula may be considered as a formal derivation of the free moment-cumulant formula \eqref{FreeMomCum}. Observe that for a fixed $n$, when $\pi=1_n$ we have that $\partial r_{1_n}(a_1, \dots, a_n)=r^\prime_n(a_1, \dots, a_n)$, and this is the only place where the term $r^\prime_n$ appears on the right-hand side of \eqref{InfFreeMomCum} (if $\pi\in \PP(n)$ and $\pi\neq 1_n$, then each block $V\in\pi$ satisfies that $|V|<n$). This means that the infinitesimal cumulants are well defined, since we can express the (unique) $r^\prime_n$ in terms of the previous cumulants $\{r_k^\prime:\A^{k}\to \mathbb{C}\}_{k=1}^{n-1}$ and the infinitesimal moments $\{\varphi_k^\prime:\A^{k}\to \mathbb{C}\}_{k=1}^n$. 

\begin{remark}
Now we can merge our free cumulants and infinitesimal free cumulants into the \textit{$\mathbb{G}$-valued free cumulants} $\{\tilde{r}_n\}_{n\geq 1}$ (see Notation \ref{tilde.notation}) which are $\cc$-linear functions from $\AA^n$ to $\ggg$. Observe that from definitions \eqref{InfFreeMomCum} and \eqref{FreeMomCum}, we directly get the following formula relating $\tilde{\varphi}_n$ with $\tilde{r}_n$:
\begin{equation}
\label{GrassFreeMomCum}
	\tilde{\varphi}_n(a_1, \dots, a_n)=\sum_{\sigma \in \NC(n)} \tilde{r}_\sigma(a_1, \dots, a_n),
\end{equation}
Thus, we could have also defined this function first, and then deduce the definition of infinitesimal cumulants by using the relation that appears when we just focus on the order-$\hbar$ coefficient. On the other hand, we can use M\"obius inversion in the lattice of non-crossing partitions to invert \eqref{GrassFreeMomCum} and thus obtain a formula that expresses $\tilde{r}$ in terms of $\tilde{\varphi}$:

\begin{equation}
\label{GrassFreeCumMom}
	\tilde{r}_n(a_1, \dots, a_n) = \sum_{\pi \in \NC(n)} \text{M\"ob}(\pi, 1_n) \tilde{\varphi}_\pi(a_1, \dots, a_n).
\end{equation}
Hence, if we focus on the order-$\hbar$ coefficient, we arrive at the following explicit description (which may as well serve as the definition) of the infinitesimal free cumulants in terms of the functionals $\varphi$ and $\varphi^\prime$:
\begin{equation}
\label{imcf1}
	r_n^\prime(a_1,\ldots,a_n) = \sum_{\pi\in \NC(n)} \text{M\"ob}(\pi,1_n) \partial\varphi_\pi(a_1,\ldots,a_n).
\end{equation} 
\end{remark}

{
\begin{remark}\label{rmk:ingGenSer}
A natural question is if it is possible to describe the multivariate infinitesimal free moment-cumulant relations in terms of generating series. Recall that for the free moment-cumulant relation we already have Equation \eqref{eq.mom.cum.gen.series.rel}. So, intuitively, a candidate relation could be its formal derivative, namely
\begin{equation}
M'_a(z_1, \dots, z_k) = \partial \left( R_a \Big(z_1(1+M_a(z_1, \dots, z_k)), \dots, z_k(1+M_a(z_1, \dots, z_k))\Big) \right),
\end{equation}
where on the left-hand side the symbol $\partial$ accounts for some sort of partial derivative on formal power series with complex coefficients in non-commuting indeterminates. The problem with this approach is that it requires a non-commutative chain rule. The basic ways of making sense of combining non-commutative differentiation and a chain rule do not work straightforwardly in this case, namely, they do not produce a valid formula.  In general, expressing $M'_a$ in terms of $M_a$, $R_a$ and $R'_a$  appears to be a rather intricate problem of which we do not know the solution. However, the shuffle algebra approach provides an equivalent result. The reader is referred to Equation \eqref{freeShuffleEq} and its shuffle algebra solution \eqref{inffreerel} in Subsection \ref{ssec:FreeBoolShuff} below. 
\end{remark}
}



\subsection{Infinitesimal Boolean cumulants}
\label{ssec:infboolean}

We follow the ideas of the foregoing subsection to define infinitesimal Boolean cumulants.

\begin{definition}
	Let $(\A,\varphi,\varphi^\prime)$ be an incps and let $\{ b_n\colon \A^{n}\to \mathbb{C}\}_{n \ge 1}$ be the corresponding Boolean cumulant functionals. The \textit{infinitesimal Boolean cumulants} are the family of multilinear functionals $\{b_n^\prime \colon \A^n\to \mathbb{C} \}_{n \ge 1}$ recursively defined by the infinitesimal moment--cumulant formula:
\begin{equation}
\label{InfBooMomCum}
	\varphi^\prime_n(a_1, \dots, a_n)
	= \sum_{\pi \in \II(n)} \partial  b_\pi(a_1, \dots, a_n) 
	\qquad \forall n\geq 1,\ a_1,\ldots,a_n\in \AA.
\end{equation}
\end{definition}

{
\begin{remark}
\label{remHas}
The same objects where defined by Takahiro Hasebe in the work \cite{Has} under the name of differential cumulants, together with higher order differential cumulants. The approach in the paper consists in considering formal power series valued linear mappings $\varphi^t:\A\to \mathbb{C}[[t]]$ and to define the notion of differential independence according to the usual rules for natural independence in the context of power series. Since in this paper we restrict our attention to only first order differential cumulants, we opted for a simplified presentation of this notions rather than the one used in \cite{Has}. We also remark that the notion of infinitesimal Boolean independence is defined in \cite{Has} and it is shown that it is equivalent to the vanishing of the mixed Boolean cumulants and mixed infinitesimal Boolean cumulants (see Appendix \ref{ssec:independences}).
\end{remark}
}

\begin{remark}
As before, we may consider $\tilde{b}_n: \AA^n\to\ggg$. Then, equations \eqref{InfBooMomCum} and \eqref{BooMomCum} directly imply the following formula relating $\tilde{\varphi}_n$ with $\tilde{b}_n$:
	\begin{equation}
    \label{GrassBooMomCum}
    \tilde{\varphi}_n(a_1, \dots, a_n)=\sum_{\sigma \in \II(n)} \tilde{b}_\sigma(a_1, \dots, a_n).
    \end{equation}
As expected, by inverting \eqref{GrassBooMomCum} in the lattice of interval partitions, we can express $\tilde{b}_n$ in terms of $\tilde{\varphi}$:

\begin{equation}
\label{GrassBooCumMom}
	\tilde{b}_n(a_1, \dots, a_n) = \sum_{\pi \in \II(n)} (-1)^{|\pi|-1} \tilde{\varphi}_\pi(a_1, \dots, a_n).
\end{equation}
\end{remark}
Finally, if we just focus on the order-$\hbar$ coefficient we get an equivalent definition of infinitesimal Boolean cumulants. This formula has the advantage of providing an explicit description of the infinitesimal Boolean cumulants in terms of the functionals $\varphi$ and $\varphi^\prime$:

\begin{equation}
\label{InfBooCumMom}
	b^\prime_n(a_1, \dots, a_n) := \sum_{\pi \in \II(n)} (-1)^{|\pi|-1} \partial\varphi_\pi(a_1, \dots, a_n).
\end{equation}
{{
The particularly simple relation \eqref{boolean.gen.series.rel} between the multivariate moment and Boolean cumulant generating series permits its extension to the infinitesimal setting as it involves only the non-commutative product rule. 
For a tuple $a = (a_1,\ldots,a_k)$ of elements in $\A$, we introduce the \textit{multivariate infinitesimal moment series} of $a$:
$$
	M^\prime_{a}(z_1, \dots, z_k)
	=\sum_{n=1}^\infty \sum_{i_1, \dots, i_n=1}^k \varphi^\prime_n(a_{i_1},\dots, a_{i_n}) z_{i_1} \cdots z_{i_n},
$$
where the $\{\varphi_n^\prime\}_{n\geq1}$ are the infinitesimal moments, and the \textit{multivariate $\eta^\prime$-series} of $a$ is 
$$
	B^\prime_{a}(z_1, \dots, z_k)
	=\sum_{n=1}^\infty \sum_{i_1, \dots, i_n=1}^k b^\prime_n(a_{i_1},\dots, a_{i_n}) z_{i_1} \cdots z_{i_n},
$$
where the $\{b_n^\prime\}_{n\geq1}$ are the infinitesimal Boolean cumulants of $\tilde{\varphi}$.}

\begin{theorem}
\label{thm:infBoolGenSer}
Let $(\A,\varphi,\varphi')$ be an incps and consider $a = (a_1,\ldots, a_k) \in \A^k$. Then we have the following relation:
\begin{equation}
\label{InfBooCumMomGenSer}
	M^\prime_{a}(z_1, \dots, z_k) 
	= (1+M_{a}(z_1, \dots, z_k))B_{a}^\prime(z_1, \dots, z_k) (1+M_a(z_1, \dots, z_k)).
\end{equation}
\end{theorem}

\begin{proof}
Note that  \eqref{InfBooCumMomGenSer} can be shown directly. Let us denote by $\II^3_\emptyset (n)$ the set of partitions of $[n]$ into three interval blocks $(S_1,S_2,S_3)$, where we allow $S_1$ and $S_3$ to be the empty set. 
Then we can see that on both sides of the equation, the coefficient of the monomial $z_{i_1} \cdots z_{i_n}$ is equal to 
$$
	\sum_{(S_1,S_2,S_3)\in \II^3_\emptyset(n)} \pai \sum_{ \pi_1 \in \II(S_1)} b_{\pi_1}(a_{S_1})\pad  
	b^\prime_{|S_2|}(a_{S_2}) \pai \sum_{ \pi_3 \in \II(S_3)} b_{\pi_3}(a_{S_3})\pad.
$$
To see that this is the coefficient on the left-hand side, we use the infinitesimal Boolean moment-cumulant formula \eqref{InfBooMomCum} and observe that the term $b_{\pi_1}(a_{S_1}) b^\prime_{|S_2|}(a_{S_2}) b_{\pi_3}(a_{S_3})$ comes from the partition $\pi=\pi_1\cup\{S_2\}\cup \pi_3 \in \II(n)$ with special block $V=S_2$ (that gets the infinitesimal cumulant).
The fact that this sum is the coefficient on the right-hand side follows from multiplying the three series and using the Boolean moment-cumulant formula \eqref{BooMomCum}. 
\end{proof}

The generating series \eqref{InfBooCumMomGenSer} of Theorem \ref{thm:infBoolGenSer} is equivalent to a recursive relation involving the generating series $M_a(z_1, \dots, z_k)$, $M^\prime_a(z_1, \dots, z_k)$, $B_a(z_1, \dots, z_k)$, and $B_a^\prime(z_1, \dots, z_k)$.

\begin{corollary}
\label{cor:infBoolGenSer} 
Let $(\A,\varphi,\varphi')$ be an incps and consider $a = (a_1,\ldots, a_k) \in \A^k$. Then we have that:
\begin{equation}
\label{InfBooCumMomGenSerBIS}
	M^\prime_{a}(z_1, \dots, z_k)= M^\prime_{a}(z_1, \dots, z_k)B_{a}(z_1, \dots, z_k) 
		+ (1+M_{a}(z_1, \dots, z_k))B^\prime_{a}(z_1, \dots, z_k).
\end{equation}
\end{corollary}

\begin{proof}
Using $M_{a}(z_1, \dots, z_k)= (1+M_{a}(z_1, \dots, z_k))B_{a}(z_1, \dots, z_k)$, relation \eqref{InfBooCumMomGenSer} can be written as 
$$
	M^\prime_{a}(z_1, \dots, z_k)(1+M_{a}(z_1, \dots, z_k))^{-1} 
	= (1+M_{a}(z_1, \dots, z_k))B^\prime_{a}(z_1, \dots, z_k).
$$ 
A simple computation yields
$$
	M^\prime_{a}(z_1, \dots, z_k)(1 - B_{a}(z_1, \dots, z_k)) 
	= (1+M_{a}(z_1, \dots, z_k))B^\prime_{a}(z_1, \dots, z_k),
$$
which implies relation \eqref{InfBooCumMomGenSerBIS}. 
\end{proof}
{{
We remark that at the level of comparing coefficients, relation \eqref{InfBooCumMomGenSerBIS} yields the following 
\begin{eqnarray*}
\varphi^\prime_n(a_{i_1},\dots, a_{i_n}) &=& b^\prime_n(a_{i_1},\dots, a_{i_n}) 
		+ \sum_{s=1}^{n-1} \varphi'_s(a_{i_1},\ldots,a_{i_s})b_{n-s}(a_{i_{s+1}},\ldots,a_{i_n})
		\\ & &+ \sum_{s=1}^{n-1} \varphi_s(a_{i_1},\ldots,a_{i_s})b_{n-s}^\prime(a_{i_{s+1}},\ldots,a_{i_n}).
\end{eqnarray*}
Observe that this formula can be used to recursively compute the infinitesimal moments.}


\subsection{Infinitesimal monotone cumulants}
\label{ssec:infmonotone}

The same ideas apply when defining infinitesimal monotone cumulants.

\begin{definition}
	Let $(\A,\varphi,\varphi^\prime)$ be an incps, let $\{h_n\colon \A^{n}\to \mathbb{C}\}_{n \ge 1}$ be the corresponding monotone cumulant functionals. The \textit{infinitesimal monotone cumulants} are the family of multilinear functionals $\{h_n^\prime \colon \A^n\to \mathbb{C} \}_{n \ge 1}$ recursively defined by the following formula: 
\begin{equation}
\label{InfMonMomCum}
	\varphi^\prime_n(a_1, \dots, a_n) 
	= \sum_{(\pi,\lambda)\in \MM(n)} \frac{1}{|\pi|!} \partial h_\pi(a_1, \dots, a_n) 
	\qquad \forall n\geq 1,\ a_1,\ldots, a_n\in \AA.
	\end{equation}	
\end{definition}

{
\begin{remark}
Analogous to the Boolean case, infinitesimal monotone cumulants already appeared in \cite{Has} under the name of first order differential cumulants (see Remark \ref{remHas}).
\end{remark}
}

\begin{remark}
If we take $\tilde{h}_n=h_n+\hbar h^\prime_n$ (see Notation \ref{tilde.notation}), Equations \eqref{InfMonMomCum} and \eqref{MonMomCum} give us:
	\begin{equation}
	\label{GrassMonMomCum}
	\tilde{\varphi}_n(a_1, \dots, a_n)=\sum_{(\pi,\lambda)\in \MM(n)} \frac{1}{|\pi|!} \partial \tilde{h}_\pi(a_1, \dots, a_n).
	\end{equation}
Thus, we could have also defined these functions first, and then deduce the infinitesimal cumulants using the relation that appears when we just focus on the order-$\hbar$ coefficient.
\end{remark}

We note in passing that in this work, we are only interested in $\ggg$-valued $\cc$-linear functionals. However, we could replace $\ggg$ with some other commutative unital algebra $\CC$ over $\cc$. For instance, following \cite{Fev,Has} we may define a higher order infinitesimal version of free, Boolean and monotone cumulants. In this setting, \eqref{GrassBooMomCum} and \eqref{GrassMonMomCum} could be identified with the analog first order infinitesimal cumulants introduced in \cite{Fev,Has}.

\begin{remark}
Similar to the infinitesimal free case, describing multivariate infinitesimal monotone moment-cumulant relations using generating series seems to be rather intricate. We are not aware of such a result. However, the shuffle algebra approach permits to describe an equivalent result. We refer the reader to Equation \eqref{monotonecumulants} in Subsection \ref{ssec:monotone}. 
\end{remark}


\section{Infinitesimal cumulants from the shuffle viewpoint}
\label{sec:ShuffleInfCumulants}


\subsection{Extension to the commutative algebra $\mathbb{G}$}
\label{ssec:Grassmann}

We now consider the incps $(\A,\varphi,\varphi^\prime)$ and the associated double tensor Hopf algebra $H = T(T_+(\A))$ together with the character $\Phi:H\to\mathbb{C}$ in $G$ defined in terms of $\varphi$. Motivated by the previous section, we want to consider a $\mathbb{C}$-linear map 
$\tilde{\Phi}:H\to \mathbb{G} \in \operatorname{Hom}(H, \mathbb{G})$ defined by 
$$
	\tilde{\Phi}=\Phi + \hbar \Phi^\prime.
$$
Since we want to mimic the shuffle algebra approach to moment-cumulant relations in terms of $\tilde{\Phi}$, we need that this map is a character. More precisely, given two elements $w_1,w_2\in H$, then $\tilde{\Phi}(w_1|w_2) = \Phi(w_1|w_2) + \hbar \Phi^\prime(w_1|w_2)$ and
\begin{eqnarray*}
	\tilde{\Phi}(w_1)\tilde{\Phi}(w_2)
	&=& (\Phi(w_1)+\hbar \Phi^\prime(w_1))(\Phi(w_2)+\hbar \Phi^\prime(w_2))\\ 
	&=& \Phi(w_1)\Phi(w_2) 
			+ \hbar \Big(\Phi(w_1)\Phi^\prime(w_2) 
			+ \Phi^\prime(w_1)\Phi(w_2) \Big).
\end{eqnarray*}
{More generally, we say that $\Phi'$ has the \textit{Leibniz-type property} if 
\begin{equation}
\label{Multprima}
	\Phi'(w_1|w_2|\cdots | w_k) 
	= \sum_{i=1}^k \left(\prod_{\substack{j=1\\j\neq i}}^k \Phi(w_j)\right) \Phi'(w_i)
\end{equation}
for any elements $w_1,\ldots,w_k\in H$. Hence, $\tilde{\Phi}$ is {multiplicative} on $H$ if and only if $\Phi$ is character on $H$ and $\Phi^\prime$ has the Leibniz-type property \eqref{Multprima}.}
\par Thereby, we define $\Phi^\prime \colon H \to \mathbb{C}$ to be the linear extension of $
\varphi^\prime$ defined by \linebreak $\Phi^\prime(w) := \varphi^\prime(a_1 \cdot_{\scriptscriptstyle{\A}} \cdots \cdot_{\scriptscriptstyle{\A}} a_n)$ for $w = a_1\cdots a_n$ and satisfying the Leibniz-type property \eqref{Multprima}. This implies that $\Phi^\prime(\mathbf{1})=0$. We then have that $\tilde{\Phi} $ is multiplicative, $\tilde{\Phi}(w_1|w_2) = \tilde{\Phi}(w_1)\tilde{\Phi}(w_2)$, and 
$$
	\tilde{\Phi}(\mathbf{1}) 
	= \Phi(\mathbf{1}) + \hbar \Phi^\prime(\mathbf{1}) 
	= 1_{\mathbb{G}}.
$$
Hence, it is natural to consider the space $\operatorname{Hom}(H,\mathbb{G})$ of linear maps from $H$ into $\mathbb{G}$. The main point here is that all the constructions and results that we obtained for $\operatorname{Hom}(H,\mathbb{C})$ carry over to $\operatorname{Hom}(H,\mathbb{G})$. We remark that it is the commutativity of $\mathbb{G}$, which is central for this to work smoothly. In particular, for $\tilde f,\tilde g\in \operatorname{Hom}(H,\mathbb{G})$, we have the shuffle algebra products
\begin{eqnarray*}
	\tilde f\star\tilde g 
		&=& m_{\mathbb{G}}(\tilde f\otimes \tilde g)\Delta,\\
	\tilde f\prec \tilde g 
		&=& m_{\mathbb{G}}(\tilde f\otimes \tilde g)\Delta_\prec,\\ 
	\tilde f\succ \tilde g 
		&=& m_{\mathbb{G}}(\tilde f\otimes \tilde g)\Delta_\succ,
\end{eqnarray*}
where $m_\mathbb{G}$ stands for the multiplication in $\mathbb{G}$. With these operations, we then conclude that $(\operatorname{Hom}(H,\mathbb{G}),\prec,\succ)$ is a unital shuffle algebra. From now on, if we have an element $\tilde f\in \operatorname{Hom}(H,\mathbb{G})$, we will refer as $f,f^\prime$ to the unique elements in $\operatorname{Hom}(H,\mathbb{C})$ such that $\tilde{f}=f+\hbar f^\prime$. Observe that for $\tilde f,\tilde g\in \operatorname{Hom}(H,\mathbb{G})$ and $\ast\in\{\star,\prec,\succ\}$ we have that
$$
	\tilde f * \tilde g = (f+\hbar f^\prime)*(g+\hbar g^\prime)
	= f * g +\hbar(f^\prime * g + f* g^\prime). 
$$
Also, if $\tilde f$ is invertible in $\operatorname{Hom}(H,\mathbb{G})$ (with respect to the shuffle product $\star$), then its inverse is given as follows
$$
	\tilde f^{-1} = f^{-1}- \hbar (f^{-1} \star f^\prime \star f^{-1}).
$$

\par In the following, we denote the group of characters in $\operatorname{Hom}(H,\mathbb{G})$ by $\tilde G$ and its corresponding (pre-)Lie algebra of infinitesimal characters by $\tilde{\mathfrak{g}}$.  We also use the same definitions of characters and infinitesimal characters in $\operatorname{Hom}(H,\mathbb{G})$. Hence, we will have the following version for the $\mathbb{G}$-valued moment-cumulant relations.

\begin{proposition}
	Given a character $\tilde{\Phi}$ in $\tilde G$, there exist unique infinitesimal characters $\tilde{\rho}$, $\tilde{\kappa}$ and $\tilde{\beta}$ in $\tilde g$, such that $\tilde{\Phi}= \exp^\star(\tilde{\rho})$ and
\begin{eqnarray}
\label{MCGF}
	\tilde{\Phi} &=& \tilde{\epsilon} + \tilde{\kappa} \prec \tilde{\Phi},\\
\label{MCBG}
	\tilde{\Phi} &=& \tilde{\epsilon} + \tilde{\Phi}\succ\tilde{\beta},
\end{eqnarray}
	where $\tilde{\epsilon}\in \operatorname{Hom}(H,\mathbb{G})$ is defined as $\tilde{\epsilon}({\bf{1}})=1_{\mathbb{G}}$ and $\tilde{\epsilon}(w)=0$ if $w\not\in H\backslash\mathbb{C}\bf{1}$.
\end{proposition}
Equations \eqref{MCGF} and \eqref{MCBG} are equivalent to the systems
\begin{align}
	\Phi &= \epsilon + \kappa \prec \Phi, \qquad \Phi^\prime 
		= \kappa^\prime \prec \Phi + \kappa\prec \Phi^\prime, 	\label{freeShuffleEq}\\ 
	\Phi &= \epsilon + \Phi \succ \beta, \qquad \Phi^\prime  
		= \Phi\succ \beta^\prime + \Phi^\prime \succ \beta. 		\label{BoolShuffleEq}
\end{align}
Recall that $\Phi$ is a character and $\Phi^\prime$ satisfies a formal Leibniz rule \eqref{Multprima}. We can show that $\rho^\prime$, $\kappa^\prime$ and $\beta^\prime$ are infinitesimal characters on $\Hom(H,\mathbb{C})$.


\subsection{Free and Boolean infinitesimal moment-cumulant relation in the shuffle approach}
\label{ssec:FreeBoolShuff}

In previous works, Ebrahimi-Fard and Patras showed that fixed point equations $\Phi = \epsilon + \kappa \prec \Phi$ and $\Phi = \epsilon + \Phi \succ \beta$ are equivalent to the free and Boolean cumulant-moment relations, respectively. We will see that the order-$\hbar$ coefficients of Equations \eqref{MCGF} and \eqref{MCBG} correspond to the infinitesimal version of the moment-cumulant relations.

\begin{proposition}
\label{prop:infmomentcumulant}
	Let $(\A,\varphi,\varphi^\prime)$ be an incps, and let $\Phi,\Phi^\prime$ be the corresponding extensions to linear maps $H\to\mathbb{C}$ as above. Let $\tilde{\Phi} = \Phi + \hbar \Phi^\prime$ be a character in $\Hom(H,\mathbb{G})$, and let $\tilde{\kappa}=\kappa+\hbar \kappa^\prime$ and $\tilde{\beta}=\beta+\hbar\beta^\prime$ be the solutions of \eqref{MCGF} and \eqref{MCBG}, respectively. Then for every word $w = a_1\cdots a_n \in T_+(\A)$, we have that
$$
	\kappa^\prime(w)=r_n^\prime(a_1,\ldots,a_n) \quad\ \textrm{and}
	\quad\
	\beta^\prime(w)=b_n^\prime(a_1,\ldots,a_n).
$$
This means that the infinitesimal characters $\kappa^\prime $ and $\beta^\prime$ evaluated in a word $w = a_1\cdots a_n$ actually identify with the infinitesimal free cumulants and infinitesimal Boolean cumulants of $(a_1,\ldots,a_n)$, respectively.
\end{proposition}

\begin{proof}
	We will prove only the infinitesimal Boolean case. The proof is by induction on $n$. The base case is obvious. For the inductive step, assume that we have the result for words of length smaller than $n$. From the definition of $\Delta_\succ$ and \eqref{BoolShuffleEq} we obtain for the word $w=a_1\cdots a_n \in T_+(\mathcal{A})$ that
\begin{eqnarray*}
	\lefteqn{\varphi^\prime(a_1 \cdot_{\scriptscriptstyle{\A}} \cdots \cdot_{\scriptscriptstyle{\A}} a_n) 
	= \Phi^\prime(w) = \Phi\succ \beta^\prime(w) + \Phi^\prime \succ \beta(w)}\\
	 &=& \sum_{1\not\in S\subseteq [n]} \Phi(a_S)\beta^\prime(a_J^S) 
	 	+ \sum_{1\not\in S\subseteq [n]} \Phi^\prime(a_S)\beta(a_J^S)\\ 
	 &=& \sum_{m=1}^{n} \beta^\prime(a_1\cdots a_m)\Phi(a_{m+1}\cdots a_n) 
	 	+ \sum_{m=1}^{n-1} \beta(a_1\cdots a_m)\Phi^\prime(a_{m+1}\cdots a_n),
\end{eqnarray*}
where in the last equation we used that both $\beta$ and $\beta^\prime$ are infinitesimal characters. Using the Boolean moment-cumulant relation, we note that the first sum above is equal to
\begin{align*}
	\lefteqn{\sum_{m=1}^{n} \beta^\prime(a_1\cdots a_m)\Phi(a_{m+1}\cdots a_n) 
	= \beta^\prime(a_1\cdots a_n) }\\
	&	+ \sum_{m=1}^{n-1} \beta^\prime(a_1\cdots a_m)\sum_{\pi\in \II(\{m+1,\ldots,n\})}  \prod_{V\in\pi} \beta(a_V) \\ 
	&= \sum_{\pi\in \II(n)} \beta^\prime(a_{V_1}) \prod_{W\in\pi \atop W\neq V} \beta(a_W),
\end{align*}
where $V_1$ denotes the block in $\pi$ containing 1.  On the other hand, by using the induction hypothesis in the second sum we have that
\begin{eqnarray*}
	\lefteqn{\sum_{m=1}^{n-1} \beta(a_1\cdots a_m)\Phi^\prime(a_{m+1}\cdots a_n)}  \\
	&=& \sum_{m=1}^{n-1} \beta(a_1\cdots a_m) \sum_{\pi\in \II(\{m+1,\ldots,n\})}
		\sum_{V\in \pi} \beta^\prime(a_V) 
		\prod_{W\in\pi \atop W\neq V} \beta(a_W)\\ 
	&=&  \sum_{\pi\in \II(n)}\sum_{V \in \pi \atop V\neq V_1} \beta^\prime(a_V) 
		\prod_{W\in\pi \atop W\neq V} \beta(a_W).
\end{eqnarray*}
Combining the two sums above we obtain
\begin{equation}
\label{infboolean}
	\varphi^\prime(a_1 \cdot_{\scriptscriptstyle{\A}} \cdots \cdot_{\scriptscriptstyle{\A}} a_n) 
	= \sum_{\pi\in \II(n)}\sum_{V\in \pi} \beta^\prime(a_V)
	\prod_{W\in \pi \atop W\neq V}\beta(a_W),
\end{equation}
and since $\beta(w)=b(a_1, \ldots, a_n)$, this permits to identify the infinitesimal character $\beta^\prime$ with the multivariate infinitesimal Boolean cumulants, i.e., $ \beta^\prime(w)=b^\prime(a_1, \ldots, a_n)$.
\end{proof}

Let us return to \eqref{addconv} in the context of $\tilde G$ and its Lie algebra $\tilde{\mathfrak{g}}$. For a $\mathbb{G}$-valued character $\tilde\Phi \in \tilde G$ and $\tilde\kappa,\tilde\beta \in \tilde{\mathfrak{g}}$, we observe that 
\begin{align*}
	\tilde\Phi 
	= \mathcal{E}_\prec(\tilde\kappa) 
	=  \mathcal{E}_\prec(\kappa+\hbar \kappa^\prime) 
	&= \Phi \star (\mathcal{E}_\prec(\hbar\kappa^\prime)\ \bole \Phi)\\
	&=\mathcal{E}_\prec(\kappa) \star \mathcal{E}_\prec(\hbar \theta_{\Phi}\kappa^\prime)\\
	&= \Phi\star(\varepsilon + \hbar \theta_{\Phi}(\kappa^\prime))\\
	&= \Phi + \hbar \Phi \star\theta_{\Phi}(\kappa^\prime).
\end{align*}
Here we used that 
$$
	\mathcal{E}_\prec(\hbar \theta_{\Phi}\kappa^\prime)
	=\varepsilon + \hbar \theta_{\Phi}(\kappa^\prime) + O(\hbar^2).
$$
From this we deduce that
\begin{equation}
\label{inffreerel}
	\Phi^\prime=\Phi \star\theta_{\Phi}(\kappa^\prime),
\end{equation}
which is the (shuffle algebra) solution of the defining fixed point equation $\Phi^\prime = \kappa^\prime \prec \Phi + \kappa\prec \Phi^\prime$ and therefore gives the infinitesimal free moment-cumulant relations. See \cite{EM} for more details on shuffle algebra equations. Moreover, it implies that $\Phi^{-1} \star \Phi^\prime \in {\mathfrak{g}}$ and that 
$$
	\kappa^\prime=\theta_{\Phi^{-1}}(\Phi^{-1} \star \Phi^\prime)
	=\Phi\succ (\Phi^{-1} \star \Phi^\prime) \prec \Phi^{-1} .
$$

Let us do some explicit computations. Recall that $\Phi = \epsilon + \kappa \prec \Phi$ implies that 
$$
	\mathcal{E}_\prec(\kappa) 
	= \varepsilon + \kappa + \kappa \prec \kappa + \kappa \prec (\kappa \prec \kappa) + \cdots
$$
Then we find in the univariate case, i.e., for a single letter $a \neq \mathbf{1}$:
\begin{align*}
	\Phi^\prime(a) 
	&= \Phi \star\theta_{\Phi}(\kappa^\prime)(a)=\kappa^\prime(a)=r^\prime_1\\
	\Phi^\prime(aa) 
	&= \theta_{\Phi}(\kappa^\prime)(aa) + 2 \Phi(a)\theta_{\Phi}(\kappa^\prime)(a)\\
	&= \kappa^\prime(aa) + 2\kappa(a)\kappa^\prime(a)\\
	&=r^\prime_2 + 2 r^\prime_1r_1\\
	\Phi^\prime(aaa) 
	&=  \theta_{\Phi}(\kappa^\prime)(aaa) 
		+ 2 \Phi(a)\theta_{\Phi}(\kappa^\prime)(aa) 
		+  3\Phi(aa)\theta_{\Phi}(\kappa^\prime)(a)\\
	&= \kappa^\prime(aaa) + 3\kappa^\prime(aa)\kappa(a) 
		+ 3 \kappa(aa)\kappa^\prime(a) + 3 \kappa(a)\kappa(a)\kappa^\prime(a)  \\ 
	&= r^\prime_3 + 3r^\prime_2r_1 + 3 r^\prime_1r_2 + 3 r^\prime_1r_1r_1 \\	
	\Phi^\prime(aaaa) 
	&= \theta_{\Phi}(\kappa^\prime)(aaaa) + 2 \Phi(a) \theta_{\Phi}(\kappa^\prime)(aaa) 
		+ 3 \Phi(aa) \theta_{\Phi}(\kappa^\prime)(aa) 
		+ 4 \Phi(aaa) \theta_{\Phi}(\kappa^\prime)(a) \\
	&= \kappa^\prime(aaaa) 
		+ 4 \kappa^\prime(aa)\kappa(aa) 
		+ 4 \kappa^\prime(aaa)\kappa(a)
		+ 4 \kappa(aaa)\kappa^\prime(a)\\
	&\qquad
		+ 6 \kappa(a)\kappa(a) \kappa^\prime(aa)
		+ 12\kappa(aa)\kappa(a)\kappa^\prime(a) 
		+ 4 \kappa(a)\kappa(a)\kappa(a)\kappa^\prime(a)\\
	&= r^\prime_4 + 4 r^\prime_2r_2 + 4r^\prime_3r_1 + 4 r^\prime_1r_3 
		+ 6 r^\prime_2r_1r_1 + 12 r^\prime_1r_2r_1 + 4 r^\prime_1r_1r_1r_1. 	
\end{align*}

\begin{remark}
{{
	In the Hopf algebra setting, the multivariate case is automatically included since we can make the computations on an arbitrary word $w \in T_+(\mathcal{A})$. For instance
	\begin{eqnarray*}
	\Phi'(a_1a_2a_3) 
	&=& \theta_{\Phi}(\kappa')(a_1a_2a_3) 
		+ \Phi(a_1) \theta_{\Phi}(\kappa')(a_2a_3)
		+ \Phi(a_2) \theta_{\Phi}(\kappa')(a_1|a_3) 
		+ \Phi(a_3) \theta_{\Phi}(\kappa')(a_1a_2)\\ 
	& & + \Phi(a_1a_2) \theta_{\Phi}(\kappa')(a_3)
		+ \Phi(a_1a_3) \theta_{\Phi}(\kappa')(a_2) 
		+  \Phi(a_2a_3) \theta_{\Phi}(\kappa')(a_1)\\
	&=& \kappa'(a_1a_2a_3) 
		+ \kappa'(a_1a_3)\kappa(a_2) 
		+ \kappa'(a_2a_3)\kappa(a_1) 
		+ \kappa'(a_1a_2)\kappa(a_3)\\ 
	& & + \kappa'(a_3)(\kappa(a_1a_2) 
		+ \kappa(a_1)\kappa(a_2)) 
		+  \kappa'(a_2)(\kappa(a_1a_3) 
		+ \kappa(a_1)\kappa(a_3))\\ 
	& & + \kappa'(a_1)(\kappa(a_2a_3) 
		+ \kappa(a_2)\kappa(a_3)).
	\end{eqnarray*}}}
\end{remark}

\begin{remark}
We can address in an analogous way the Boolean case. For this, we find that 
\begin{align*}
	\tilde\Phi 
	= \mathcal{E}_\succ(\tilde\beta)
	= \mathcal{E}_\succ(\beta+\hbar\beta^\prime)
	&=(\mathcal{E}_\succ(\hbar\beta^\prime)\ \bori \Phi) \star \Phi\\
	&= \mathcal{E}_\succ(\hbar \theta_{\Phi^{-1}}(\beta^\prime)) \star \mathcal{E}_\succ(\beta) \\
	&= (\varepsilon + \hbar \theta_{\Phi^{-1}}(\beta^\prime))\star\Phi\\
	&= \Phi + \hbar \theta_{\Phi^{-1}}(\beta^\prime)\star\Phi.
\end{align*}
This yields
\begin{equation}
\label{infbooleanrel}
	\Phi^\prime= \theta_{\Phi^{-1}}(\beta^\prime) \star \Phi.
\end{equation}
Expression \eqref{infbooleanrel} solves the shuffle fixed point equation $\Phi^\prime = \Phi\succ \beta^\prime + \Phi^\prime \succ \beta$ and therefore gives the infinitesimal Boolean moment-cumulant relations, i.e., Equation \eqref{infboolean} in the proof of Proposition \ref{prop:infmomentcumulant}. 
\end{remark}

\subsection{Infinitesimal monotone cumulants in the shuffle approach}
\label{ssec:monotone}

We now show how to obtain the corresponding monotone infinitesimal moment--cumulant relations from the exponential bijection given by the convolution product $\star$.

\begin{proposition}
\label{prop:MonotoneShuffle}
Let $(\mathcal{A},\varphi,\varphi^\prime)$ be an incps, and let $\Phi,\Phi^\prime$ be the corresponding extensions to linear maps $H \to \mathbb{C}$ as above. Let $\tilde{\Phi} = \Phi + \hbar \Phi^\prime$ be the character in $\operatorname{Hom}(H,\mathbb{G})$ and $\tilde{\rho}=\rho+\hbar \rho^\prime$ the infinitesimal character such that $\tilde{\Phi} = \exp^\star(\tilde{\rho})$. Then, for every word $w=a_1\cdots a_n\in T_+(\A)$, the infinitesimal character $\rho^\prime$ evaluated in a word $w = a_1\cdots a_n$ identifies with the infinitesimal monotone cumulant, i.e., $h^\prime_n(a_1,\ldots,a_n) = \rho^\prime(w)$.
\end{proposition}

\begin{proof}
By the definition of the exponential map with respect to the convolution (shuffle) product, $\star$, in the $\mathbb{G}$-valued case, we have that
\begin{eqnarray}
	\tilde{\Phi} 
	= \exp^\star(\rho + \hbar \rho^\prime)
	&=& \sum_{k=0}^\infty \frac{(\rho + \hbar \rho^\prime)^{\star k}}{k!} \nonumber \\ 
	&=& \sum_{k=0}^\infty \frac{1}{k!}\left(\rho^{\star k} 
		+ \hbar \left(\sum_{m=1}^k \rho^{\star (m-1)}\star \rho^\prime \star \rho^{\star (k-m)} \right) \right). \label{stepA}
\end{eqnarray}
We will compute the evaluation on a word $v=a_1 \cdots a_m \in T_+(\A)$ of the right-hand side in the above equation. From \cite{EP2} we conclude, since $\rho$ and $\rho^\prime$ are infinitesimal characters, that the convolution product between them only requires the ``reduced linearised'' part of the coproduct, which is given by
$$
	\overline{\Delta}(a_1\cdots a_m) 
	:= \sum_{a_{I_1} a_{I_2} a_{I_3} = v
	\atop  I_1\sqcup I_3 ,{I_2}\neq \emptyset } a_{I_1}a_{I_3} \otimes a_{I_2}.
$$
Then we have, for instance, that $\rho\star \rho^\prime (v) = m_\mathbb{C}(\rho\otimes \rho^\prime)\overline{\Delta}(v)$. More generally, it was shown in Lemma 3 of \cite{EP2} that if $\overline{\Delta}^{[q]}:T_+(A)\to T_+(A)^{\otimes (q+1)}$ stands for the $q$-fold left iterated reduced linearised coproduct, i.e., $\overline{\Delta}^{[q]} = (\overline{\Delta}^{[q-1]}\otimes \operatorname{id})\overline{\Delta}$, then
\begin{equation}
	\overline{\Delta}^{[q-1]}(a_1\cdots a_m) 
	= \sum_{\substack{\pi \in \mathcal{M}^q(m)\\ \pi = \{V_1,\ldots,V_q\}}} a_{V_1}\otimes \cdots \otimes a_{V_q},
\end{equation}
where $\mathcal{M}^q(m)$ denotes the set of monotone partitions of $[m]$ into $q$ blocks, and the blocks of $\pi$ are naturally pre-ordered.
\par We proceed now to the evaluation of the order-$\hbar$ coefficient on the right-hand side of \eqref{stepA}. Given a $k\geq 1$ and $1\leq m\leq k$ we have that 

\begin{eqnarray*}
	\frac{1}{k!}\rho^{\star(m-1)} \star \rho^\prime \star \rho^{{k-m}}(a_1\cdots a_n) 
	&=& \frac{1}{k!} m_\mathbb{C}^{[k]}(\rho \otimes \cdots \otimes  \rho^\prime \otimes \cdots\otimes\rho)
	\overline{\Delta}^{[k-1]}(a_1\cdots a_n)\\ 
	&=& \frac{1}{k!}\sum_{\substack{\pi \in \mathcal{M}^k(n)\\ 
	\pi = \{V_1,\ldots,V_k\}}} \rho(a_{V_1}) \cdots \rho^\prime(a_{V_m})\cdots  \rho(a_{V_k}).
\end{eqnarray*}
Adding over $1\leq m\leq k$ as well as over $1\leq k\leq n$ we obtain
\begin{eqnarray*}
	\Phi^\prime(w) 
	&=& \sum_{k=1}^n \frac{1}{k! }\sum_{m=1}^k \rho^{\star(m-1)} 
			\star \rho^\prime \star \rho^{{k-m}}(a_1\cdots a_n) \\ 
	& =& \sum_{k=1}^n \frac{1}{k!}\sum_{\pi \in\NC^k(n)}m(\pi)\sum_{V \in \pi} \rho^\prime(a_V) 
			\prod_{\substack{W\in \pi\\W\neq V}} \rho(a_W),
\end{eqnarray*}
where  $\NC^k(n)$ stands for the set of non-crossing partitions of $[n]$ with exactly $k$ blocks and  $m(\pi)$ denotes the number of monotone labelings of $\pi$. See Appendix \ref{ssec:parititions} for a more careful discussion on $m(\pi)$. Finally, from \eqref{monotone.labellings} we get that
\begin{equation}
	\varphi^\prime_n(a_1, \cdots, a_n) 
	= \sum_{\pi \in \NC(n)} \frac{1}{\tau(\pi)!}\sum_{V \in \pi} \rho^\prime(a_V) 
	\prod_{\substack{W\in \pi\\W\neq V}} \rho(a_W), 
\end{equation}
and since $\rho(w)=h_n(a_1,\ldots,a_n)$, we conclude that $\rho^\prime(w)=h^\prime_n(a_1,\ldots,a_n)$.
\end{proof}

Consider $\tilde{\Phi}$ and $\tilde{\rho}$ as in the above proposition. This means that 
$$
	\tilde\Phi=\exp^\star(\tilde{\rho})=\exp^\star(\rho + \hbar\rho^\prime).
$$
One can show the factorisation \cite{Reu}
$$
	\exp^\star(\rho + \hbar\rho^\prime)=\exp^\star(\rho) \star \exp^\star(F(\rho,\hbar\rho^\prime)),
$$ 
where
$$
	F(\rho,\hbar\rho^\prime) 
	:= \hbar\rho^\prime + \sum_{n>0} \frac{{(-1)}^n}{(n+1)!}\mathrm{Ad}^{(n)}_{\rho}(\hbar\rho^\prime).	
$$
We can compactly write
$$
	F(\rho,\hbar\rho^\prime)
	= \frac{e^{-\mathrm{Ad}_\rho} - 1}{-\mathrm{Ad}_\rho}(\hbar \rho^\prime)
	=:W_{-\rho}(\hbar \rho^\prime).
$$
Note that $\mathrm{Ad}^{(n)}_\alpha(\beta):=[\alpha,\mathrm{Ad}_\alpha^{(n-1)}(\beta)]=\alpha \star \mathrm{Ad}_\alpha^{(n-1)}(\beta) - \mathrm{Ad}_\alpha^{(n-1)}(\beta)\star \alpha$. This then yields
$$
	\exp^\star(\rho) \star \exp^\star(F(\rho,\hbar\rho^\prime)) 
	= \exp^\star(\rho) \star(\varepsilon + \hbar W_{-\rho}(\rho^\prime)).
$$ 
Note that $W_{-\rho}(\rho^\prime)$ is an infinitesimal character. For $\exp^\star(\tilde{\rho})=\tilde\Phi=\Phi + \hbar \Phi^\prime$ this implies the shuffle analog of monotone infinitesimal moment-cumulant relations
\begin{equation}
\label{monotonecumulants}
	\Phi^\prime=\Phi \star W_{-\rho}(\rho^\prime).
\end{equation}


We compute the first three univariate monotone infinitesimal moment-cumulant relations. Let $a \in \A\subset  T_+(\mathcal{A})$. Then, thanks to $\Delta(a)=a\otimes \mathbf{1} + \mathbf{1} \otimes a$, we have
\begin{align*}
	\Phi^\prime(a)		
		&=\Phi \star W_{-\rho}(\rho^\prime)(a)
		  = W_{-\rho}(\rho^\prime)(a)
		  = \rho^\prime(a) = h^\prime_1\\[0.2cm]
	\Phi^\prime(a^2)		
		&=\Phi \star W_{-\rho}(\rho^\prime)(a^2)
			=2 \Phi(a)W_{-\rho}(\rho^\prime)(a) 
			+ W_{-\rho}(\rho^\prime)(a^2)\\
		&=2 \rho(a)\rho^\prime(a) + \rho^\prime(a^2) 
				- \frac{1}{2}[\rho,\rho^\prime](a^2)\\
		&=2 \rho(a)\rho^\prime(a) + \rho^\prime(a^2) 
				- \frac{1}{2}(\rho\star\rho^\prime - \rho^\prime\star\rho)(a^2)\\
		&= \rho^\prime(a^2) + 2 \rho(a)\rho^\prime(a) 
			= h^\prime_2 + 2 h_1 h^\prime_1\\[0.2cm]
	\Phi^\prime(a^3)	
		&=\Phi \star W_{-\rho}(\rho^\prime)(a^3)\\
		&=2 \Phi(a)W_{-\rho}(\rho^\prime)(a^2) 
			+ 3\Phi(a^2)W_{-\rho}(\rho^\prime)(a) 
				+ W_{-\rho}(\rho^\prime)(a^3)\\	
		&=\rho^\prime(a^3)
			+ \frac{5}{2}\rho^\prime(a)\rho(a^2) 
			+ \frac{5}{2} \rho(a)\rho^\prime(a^2) 
			+ 3\rho(a)\rho(a)\rho^\prime(a)\\[0.2cm]
	\Phi^\prime(a^4) 
	&= \rho^\prime(a^4) 
		+ 3\rho^\prime(a)\rho(a^3)
			+ 3\rho(a)\rho^\prime(a^3) 
				+ 3\rho(a^2)\rho^\prime(a^2) \\ 
	& +\frac{13}{3}\rho^\prime(a^2)\rho(a)\rho(a) 
		+ \frac{26}{3} \rho(a^2)\rho^\prime(a)\rho(a) 
			+ 4\rho^\prime(a)\rho(a)\rho(a)\rho(a).
\end{align*}
These relations are consistent with the formal derivation viewpoint on the monotone moment-cumulant relations:
\begin{align*}
    \Phi(a)	   &= \rho(a)\\
    \Phi(a^2) &=\rho(a^2) + \rho(a)\rho(a)\\
    \Phi(a^3) &=\rho(a^3) + \frac{5}{2} \rho(a)\rho(a^2) + \rho(a)\rho(a)\rho(a)\\
    \Phi(a^4) &=\rho(a^4) + 3 \rho(a)\rho(a^3) + \frac{3}{2} \rho(a^2)\rho(a^2) 
    			+ \frac{13}{3} \rho(a)\rho(a)\rho(a^2)+ \rho(a)\rho(a)\rho(a)\rho(a)
\end{align*}

\begin{remark}
Eventually, Equation \eqref{monotonecumulants} implies that 
$$
	W_{-\rho}(\rho^\prime)= \Phi^{-1}\star\Phi^\prime.
$$
This result is consistent with the viewpoint of taking formal derivations. From the last equation we deduce that 
\begin{equation}
\label{infmonocumulants}
	\rho^\prime 
	= W^{\circ -1}_{-\rho}(\Phi^{-1}\star\Phi^\prime) 
	= \frac{-\mathrm{Ad}_\rho}{e^{-\mathrm{Ad}_\rho} - 1}(\Phi^{-1}\star\Phi^\prime),
\end{equation}
which expresses infinitesimal monotone cumulants in a rather non-trivial manner in terms of moments and infinitesimal moments (recall that $\rho = \log^\star(\Phi)$).
\end{remark}

We can collect the obtained equations  \eqref{inffreerel}, \eqref{infbooleanrel}, \eqref{monotonecumulants} in the following result:

\begin{proposition}
Let $(\A,\varphi,\varphi^\prime)$ be a incps, and let $\Phi, \Phi^\prime$ be the corresponding extensions to linear maps $H\to \mathbb{C}$ as above. Consider the pairs of infinitesimal characters $(\kappa,\kappa^\prime),(\beta,\beta^\prime)$ and $(\rho,\rho^\prime)$  described in Propositions \ref{prop:infmomentcumulant} and \ref{prop:MonotoneShuffle}. Then we have
\begin{eqnarray*}
	\Phi^\prime
	&=&\Phi \star \theta_{\Phi}(\kappa^\prime)\\
	&=& \theta_{\Phi^{-1}}(\beta^\prime) \star \Phi\\
	&=& \Phi \star W_{-\rho}(\rho^\prime).
\end{eqnarray*}
\end{proposition}


\section{Infinitesimal cumulant-cumulant relations}
\label{sec:RelationsInfCum}

In this section, we are going to generalise to the infinitesimal setting some relations between cumulants that are already known in the context of non-commutative probability. Specifically, we want to check how the formulas \eqref{BooFree}, \eqref{BooMon}, \eqref{FreeBoo}, and \eqref{FreeMon} from Arizmendi et al, reference \cite{AHLV}, are carried over to the infinitesimal setting. The nice thing about these specific formulas is that their proofs do not really use the fact that the $\varphi_n, b_n,r_n,h_n$ are functionals and we can replace them by our $\cc$-linear functionals into $\ggg$, $\tilde{\varphi}_n, \tilde{b}_n,\tilde{r}_n,\tilde{h}_n$. We have the following theorem:

\begin{theorem}
Let $(\A,\varphi,\varphi^\prime)$ be an incps, and let $\{ \tilde{r}_n\}_{n \ge 1}$, $\{ \tilde{b}_n\}_{n \ge 1}$, and $\{ \tilde{h}_n\}_{n \ge 1}$ be the families of $\cc$-linear maps from $\A^{n}$ to $\mathbb{G}$ obtained from merging cumulants $r_n$, $b_n$, $h_n$ with infinitesimal cumulants $r^\prime_n$, $b^\prime_n$, $h^\prime_n$, respectively. Then, the following relations between $\tilde{b}$, $\tilde{r}$, $\tilde{h}$ hold:
\begin{eqnarray}
	\tilde{b}_n (a_1, \dots, a_n) 
		&=& \sum_{\pi\in \NN\CC_{irr}(n)} \tilde{r}_\pi(a_1, \dots, a_n), \label{GrassBooFree}\\
	\tilde{r}_n (a_1, \dots, a_n) 
		&=& \sum_{\pi\in \NN\CC_{irr}(n)} (-1)^{|\pi|-1} \tilde{b}_\pi(a_1, \dots, a_n), \label{GrassFreeBoo}\\
	\tilde{b}_n (a_1, \dots, a_n) 
		&=& \sum_{\pi\in \NN\CC_{irr}(n)} \frac{1}{\tau(\pi)!}\tilde{h}_\pi(a_1, \dots, a_n), \label{GrassBooMon} \\
	\tilde{r}_n (a_1, \dots, a_n) 
		&=& \sum_{\pi\in \NN\CC_{irr}(n)} \frac{(-1)^{|\pi|-1}}{\tau(\pi)!} \tilde{h}_\pi(a_1, \dots, a_n), \label{GrassFreeMon}
\end{eqnarray}
for every $n$ and elements $a_1,\dots, a_n\in \AA$.
\end{theorem}

\begin{proof}
The proof of these formulas can be obtained by either of the both approaches we discussed in Sections \ref{sec:InfCumulants} and \ref{sec:ShuffleInfCumulants}. 
Through this proof we work with an integer $n$ and generic elements $a_1,\dots, a_n\in \AA$. We will denote $w = a_1\cdots a_n\in T_+(\A)$, $a=(a_1,\dots,a_n)\in\AA^n$ and $(a|_V)=(a_{v_1},\ldots,a_{v_k})\in\AA^k$ for a subset $V=\{{v_1}< \cdots< {v_k}\}\in [n]$.
\begin{itemize}
\item To get \eqref{GrassBooFree} we begin by proposing the candidate Boolean cumulants $c_n$ by 
$$
	c_n (a_1, \dots, a_n)
	= \sum_{\pi\in \NN\CC_{irr}(n)} \tilde{r}_\pi(a_1, \dots, a_n).
$$
Our aim is to show that $c_n=\tilde{b}_n$. Observe that we can use the Grassmann free moment-cumulant formula \eqref{GrassFreeMomCum}, to compute:
\begin{align*}
	\tilde{\varphi}_n(a) 
	&=\sum_{\sigma \in \NC(n)} \tilde{r}_\sigma(a) =\sum_{\pi \in \II(n)} \sum_{\substack{\sigma \in \NC(n)\\  \sigma\ll\pi}} \tilde{r}_\sigma(a)
	= \sum_{\pi \in \II(n)} \sum_{\substack{\sigma \in \NC(n)\\ \sigma\ll\pi}} \prod_{V\in\sigma}\tilde{r}_{|V|}(a|_V)\\
	&= \sum_{\pi \in \II(n)}  \prod_{V\in\pi} \Big( \sum_{\substack{\sigma \in \NC(V)\\ 
	\sigma\ll 1_{|V|} }}\tilde{r}_\sigma(a|_V)\Big)=\sum_{\pi \in \II(n)}  \prod_{V\in\pi} c_{|V|}(a|_V) 
	=\sum_{\pi \in \II(n)} c_{\pi}(a),
\end{align*}
where $\ll$ stands for the min-max order in $\NC(n)$ (see Definition \ref{order.NC}). Thus, we arrive at the fact that the proposed candidates actually satisfy the infinitesimal Boolean moment-cumulant formula \eqref{InfBooMomCum} for all $n$ and elements $a_1,\dots, a_n\in \AA$. Thus, we conclude that:
$$
	\tilde{\beta}_n (a_1, \dots, a_n)
	=c_n (a_1, \dots, a_n)
	= \sum_{\pi\in \NN\CC_{irr}(n)} \tilde{r}_\pi(a_1, \dots, a_n).
$$

\item Now we prove \eqref{GrassFreeBoo}. Consider the extensions to the Hopf--algebraic framework such that $\Phi$ and $\Phi^\prime$ are the elements in $\operatorname{Hom}(H,\mathbb{C})$ extending $\varphi$ and $\varphi^\prime$ as before, $\tilde{\Phi} = \Phi + \hbar \Phi^\prime$ is the $\mathbb{G}$-valued character associated, and $\tilde{\kappa}$ and $\tilde{\beta}$ are the $\mathbb{G}$-valued infinitesimal characters associated to the free and Boolean cumulants of the pair $(\varphi,\varphi^\prime)$, respectively. Consider a word $w = a_1\cdots a_n\in T_+(\A)$. From the $\mathbb{G}$-valued fixed point equations for $\tilde{\Phi}$, we have that $\tilde{\kappa} \prec \tilde{\Phi} = \tilde{\Phi}\succ \tilde{\beta}$. From the shuffle algebra axioms, the latter equation implies the relation
$$
	\tilde{\Phi}^{-1}\succ \tilde{\kappa} = \tilde{\beta} \prec \tilde{\Phi}^{-1}.
$$ 
From the $\mathbb{G}$-valued case of Equation  \eqref{InvSucc} we have that $\tilde{\Phi}^{-1} = \E_\prec(-\tilde{\beta})$. The latter equation evaluated in $w$ then yields the identity
$$
	\tilde{\Phi}^{-1}(a_1\cdots a_n) = \sum_{\pi\in\NC(n)}(-1)^{|\pi|}\tilde{b}_{\pi}(a_1,\ldots,a_n).
$$
Then we use induction to obtain
\begin{eqnarray*}
	\lefteqn{\tilde{\Phi}^{-1}\succ \tilde{\kappa}(w) 
	= \tilde{\kappa}(a_1\cdots a_n) + \sum_{j=1}^{n-1}\tilde{\kappa}(a_1\cdots a_j)\tilde{\Phi}^{-1}(a_{j+1}\cdots a_n)}\\ 
	&=&  \tilde{\kappa}(w) + \sum_{j=1}^{n-1} 
	\left(\sum_{\pi\in \NC_{irr}(j)} (-1)^{|\pi|-1} \tilde{b}_{\pi}(a_1,\ldots,a_j) \right)
	\left(\sum_{\sigma\in \NC(n-j)} (-1)^{|\sigma|} \tilde{b}_{\sigma}(a_{j+1},\ldots,a_n) \right)\\ 
	&=& \tilde{\kappa}(w) + 
	\sum_{\pi \in \mathcal{NC}(n)\backslash \NC_{irr}(n)}(-1)^{|\pi|-1} \tilde{b}_\pi(a_1,\ldots,a_n).
\end{eqnarray*}
On the other hand
\begin{eqnarray*}
	\tilde{\beta}\prec \tilde{\Phi}^{-1}(w) 
	&=& \sum_{1\in S\subseteq [n]}\tilde{\beta}(a_S)\tilde{\Phi}^{-1}(a_{J_1}|\cdots|a_{J_p})\\ 
	&=&  \sum_{1\in S\subseteq [n]}\tilde{\beta}(a_S)\prod_{j=1}^{p}
	\left(\sum_{\pi\in \NC(|J_j|)} (-1)^{|\pi|}\tilde{b}_\pi(a_{J_j}) \right)\\ 
	&=& \sum_{\pi\in\NC(n)} (-1)^{|\pi|-1} \tilde{b}_{\pi}(a_1,\ldots,a_n).
\end{eqnarray*}
Hence
$$
	\tilde{\kappa}(a_1\cdots a_n) =  \sum_{\pi\in\NC_{irr}(n)} (-1)^{|\pi|-1} \tilde{b}_{\pi}(a_1,\ldots,a_n).
$$
We conclude \eqref{GrassFreeBoo} once we identify $\tilde{r}_n(a_1,\ldots,a_n) = \tilde{\kappa}(a_1\cdots a_n)$.

\item Similarly, we can give a proof of \eqref{GrassBooMon} using the shuffle algebra approach. For this, we recall the $\mathbb{G}$-valued fixed point equation for Boolean cumulants $\tilde{\Phi} - \tilde{\varepsilon} = \tilde{\Phi}\succ \tilde{\beta}$. Let $\tilde{\rho}$ be the $\mathbb{G}$-valued infinitesimal character associated to the monotone cumulants of the pair $(\varphi,\varphi^\prime)$ and  $w = a_1\cdots a_n\in T_+(\A)$ be a word. Using the $\mathbb{G}$-valued monotone moment--cumulant relation and induction we have:
\begin{eqnarray*}
	 \lefteqn{\sum_{\pi\in\NC(n)} \frac{1}{\tau(\pi)!} \tilde{h}_\pi(a) 
	 =\tilde{\Phi}(w) 
	 = \tilde{\beta}(w)+ \sum_{j=1}^{n-1} \tilde{\beta}(a_1\cdots a_j) \tilde{\Phi}(a_{j+1}\cdots a_n)}\\	
	 &=& \tilde{\beta}(w)+ \sum_{j=1}^{n-1} 
	 \left(\sum_{\pi\in\NC_{irr}(j)} \frac{1}{\tau(\pi)!} \tilde{h}_\pi(a_1,\ldots,a_j) \right)
	 \left(\sum_{\sigma\in\NC(n-j)} \frac{1}{\tau(\sigma)!} \tilde{h}_\sigma(a_{j+1},\ldots,a_n) \right)\\ 
	 &=&\tilde{\beta}(w)
	 + \sum_{\pi \in \NC(n)\backslash \NC_{irr}(n)} \frac{1}{\tau(\pi)!} \tilde{h}_\pi(a_1,\ldots,a_n).
\end{eqnarray*}
{{Observe that we used the fact that $\frac{1}{\tau(\pi)!} \frac{1}{\tau(\sigma)!} = \frac{1}{\tau(\pi\sqcup \sigma)!},$ where $\pi \in \NC_{irr}(j),\sigma\in \NC(n-j)$, and $\pi\sqcup \sigma\in \NC(n)$ is the noncrossing partition obtained by taking the disjoint union of $\pi$ and $\sigma$.}} From above, we conclude Equation \eqref{GrassBooMon}.

\item {{Regarding \eqref{GrassFreeMon}, we can use again shuffle algebra. We first recall that 
\begin{eqnarray*}
\tilde{\Phi}^{-1}(w) = \exp^\star(\tilde{\rho})^{-1}(w) &=& \exp^\star(- \tilde{\rho})(w)
                 \\  &=& \sum_{\pi\in\NC(n)} \frac{1}{\tau(\pi)!}(-\tilde{h})_\pi(a_1,\ldots,a_n)
                 \\ &=&\sum_{\pi\in\NC(n)} \frac{(-1)^{|\pi|}}{\tau(\pi)!}\tilde{h}_\pi(a_1,\ldots,a_n) ,
\end{eqnarray*}
for any word $w = a_1\cdots a_n\in H$. On the other hand, we have that $\tilde{\Phi}^{-1} = \mathcal{E}_\prec(\tilde{\kappa})^{-1} = \mathcal{E}_\succ(-\tilde{\kappa}).$ Hence, we have the fixed-point equation 
$$
	\tilde{\Phi}^{-1} - \tilde{\epsilon} = \tilde{\Phi}^{-1}\succ (- \tilde{\kappa}).
$$
Now we can proceed in a similar way to that of proving \eqref{GrassBooMon}. For a word $w = a_1\cdots a_n \in H$, using induction we have that
\begin{eqnarray*}
	 \lefteqn{\sum_{\pi\in\NC(n)} \frac{(-1)^{|\pi|}}{\tau(\pi)!} \tilde{h}_\pi(a_1,\ldots,a_n) 
	 =\tilde{\Phi}^{-1}(w) 
	 = -\tilde{\kappa}(w)- \sum_{j=1}^{n-1} \tilde{\kappa}(a_1\cdots a_j) \tilde{\Phi}(a_{j+1}\cdots a_n)}\\	
	 &=& -\tilde{\kappa}(w)- \sum_{j=1}^{n-1} 
	 \left(\sum_{\pi\in\NC_{irr}(j)} \frac{(-1)^{|\pi|-1}}{\tau(\pi)!} \tilde{h}_\pi(a_1,\ldots,a_j) \right)
	 \left(\sum_{\sigma\in\NC(n-j)} \frac{(-1)^{|\sigma|}}{\tau(\sigma)!} \tilde{h}_\sigma(a_{j+1},\ldots,a_n) \right)\\ 
	 &=&-\tilde{\kappa}(w)
	 - \sum_{\pi \in \NC(n)\backslash \NC_{irr}(n)} \frac{(-1)^{|\pi|-1}}{\tau(\pi)!} \tilde{h}_\pi(a_1,\ldots,a_n),
\end{eqnarray*}
and hence Equation \eqref{GrassFreeMon} follows.
}}

\end{itemize}
\end{proof}

{{
\begin{remark}
We note that to get \eqref{GrassBooFree} using the shuffle approach, we can use equation $\tilde{\kappa}\prec \tilde{\Phi} = \tilde{\beta}\succ\tilde{\Phi}$ and follow arguments analogous to the proof of \eqref{GrassFreeBoo}. To show \eqref{GrassFreeBoo}, we could have followed the proof by Belinschi and Nica (See \cite{BN1}, Proposition 3.9). The proof of  \eqref{GrassBooMon} follows similar ideas to the ones used in the proof of \eqref{GrassFreeBoo}. Regarding \eqref{GrassFreeMon}, we could have followed an approach similar to how we showed \eqref{GrassBooFree}. This requires to proof some interesting combinatorial identities for rooted tree factorial in the spirit of \cite{AHLV}. 
\end{remark}}}

\begin{proof}[Proof of Theorem \ref{Thm1}]
Once we have formulas \eqref{GrassBooFree}, \eqref{GrassFreeBoo}, \eqref{GrassBooMon}, \eqref{GrassFreeMon} we can just focus on the $\hbar$ coefficient to get the respective formulas relating infinitesimal cumulants: \eqref{InfBooFree}, \eqref{InfBooMon}, \eqref{InfFreeBoo}, \eqref{InfFreeMon}.
\end{proof}

{{
\begin{remark}
At the time when the first version of this work was finalized, we did not have the complete picture of relations between infinitesimal cumulants. Indeed, the formulas expressing monotone cumulants in terms of free (or Boolean) cumulants in the multivariate case were missing. These formulas are given in the more recent work \cite{Magnus2021}, which appeared after the first release of the current paper. The extension to the infinitesimal setting can be done by considering  the $\ggg$-valued analogue of the pre-Lie Magnus expansion \eqref{PreLieMagnus0}:
$$
	\tilde{\rho} = \Omega^\prime(\tilde{\kappa}) = -\Omega^\prime(-\tilde{\beta}).
$$
The corresponding formulas between infinitesimal cumulants are
\begin{eqnarray}
	h^\prime_n (a_1, \dots, a_n) 
		&=& \sum_{\pi\in \NN\CC_{irr}(n)} \omega(\pi) \partial b_\pi(a_1, \dots, a_n),  \label{InfMonBoo}\\
	h^\prime_n (a_1, \dots, a_n) 
		&=& \sum_{\pi\in \NN\CC_{irr}(n)} (-1)^{|\pi|-1}\omega(\pi) \partial r_\pi(a_1, \dots, a_n).\label{InfMonFree}
\end{eqnarray}
Here $\omega$ is Murua's function that assigns a real value to each partition $\pi$, see Definition \ref{def.murua.coeff} in the appendix at the end.  In addition, we note that, based on the first version of this paper, a generalization of the six cumulant-cumulant formulas to the infinitesimal operator-valued case was developed in the recent work \cite{perales2020operator}.
\end{remark}}}

\section{Infinitesimal Boolean Bercovici--Pata bijection}
\label{sec:Bercovici.Pata}

The objective of this section is to apply our previous results to get an infinitesimal analogue of the work of Belinschi and Nica \cite{BN1,BN2,BN3} at the algebraic level.

{{We consider an \textit{infinitesimal law on $k$ variables}, which is a pair $\tilde{\mu}=(\mu,\mu^\prime)$  of linear functionals $\mu,\mu^\prime:\cc\langle z_1,\dots,z_k\rangle \to \cc$ such that $\mu(1) = 1$ and $\mu^\prime(1)=0$. We also define
$$
\tilde{\DD}(k) = \{\tilde{\mu} = \mu + \hbar{\mu'\,:\, (\mu,\mu')\,\mbox{ is an infinitesimal law on $k$ variables}}\}.
$$}Given $\tilde{\mu} = \mu + \hbar\mu'\in \tilde{\DD}(k)$, we can compute its $\ggg$-valued free and Boolean cumulants $\tilde{r}_n(\tilde{\mu})=r_n(\tilde{\mu})+\hbar r^\prime_n(\tilde{\mu})$, and $\tilde{b}_n(\tilde{\mu})=b_n(\tilde{\mu})+\hbar b^\prime_n(\tilde{\mu})$. With these concepts at hand, we can define an infinitesimal analogue of the \emph{Boolean Bercovici--Pata bijection} studied in \cite{BN1}.

\begin{definition}
The \textit{infinitesimal Boolean Bercovici--Pata bijection} $\tilde{\bb}:\tilde{\DD}(k)\to\tilde{\DD}(k)$ is the map $\tilde{\mu}\mapsto \tilde{\bb}(\tilde{\mu})$ such that $\tilde{\bb}(\tilde{\mu})$ is 
uniquely defined by the fact that
$$
	\tilde{r}_n(\tilde{\bb}(\tilde{\mu}))= \tilde{b}_n(\tilde{\mu}), \qquad\forall n\geq 1.
$$
\end{definition}

\begin{remark}
In the shuffle approach, the map $\tilde{\bb}$ can be simply expressed in terms of the right action, $\bole$, of $\tilde{G}$ on itself defined in \eqref{shuffleadjointgroup}. For $\tilde{\Phi} \in \tilde{G}$
$$
\tilde{\mathbb{B}}(\tilde\Phi)= \tilde\Phi\ \bole \tilde\Phi \in \tilde G.
$$
If we use the infinitesimal free and Boolean cumulants, we see that $\tilde\Phi\ \bole \tilde\Phi=\mathcal{E}_\prec(\theta_{\tilde\Phi}(\tilde\kappa))=\mathcal{E}_\prec(\tilde\beta) = \mathcal{E}_\prec(\mathcal{L}_\succ(\tilde\Phi))$. Thus we can also define $\tilde{\bb}$ by
$$
	\tilde{\mathbb{B}}(\tilde\Phi)= \mathcal{E}_\prec (\mathcal{L}_\succ(\tilde\Phi)).
$$
Recall that the left and right-shuffle logarithms are defined by $\mathcal{L}_\prec := (\tilde\Phi - \epsilon) \prec \tilde\Phi^{-1}$ respectively $\mathcal{L}_\succ := \tilde\Phi^{-1} \succ (\tilde\Phi - \epsilon)$. It is clear that $\tilde{\bb}$ has a compositional inverse $\tilde{\bb}^{-1}$ given by
$$
	\tilde{\mathbb{B}}^{-1}(\tilde\Psi)= \mathcal{E}_\succ (\mathcal{L}_\prec(\tilde\Psi))=\tilde\Psi \ \bori \tilde\Psi,
$$
where for the second equality we can follow similar steps.
\end{remark}

Now we want to define an infinitesimal version of the $\bb_t$-transform studied in \cite{BN2}. First we recall the infinitesimal free additive convolution and define its Boolean analogue on $\tilde{\DD}(k)$. 

\begin{definition}
Given $\tilde{\mu},\tilde{\nu}\in \tilde{\DD}(k)$, the \textit{infinitesimal free additive convolution of $\tilde{\mu}$ and $\tilde{\nu}$} is the infinitesimal distribution $\tilde{\mu}\boxplus\tilde{\nu}\in \tilde{\DD}(k)$ uniquely determined by the fact that
$$
	\tilde{r}_n(\tilde{\mu}\boxplus \tilde{\nu})=\tilde{r}_n(\tilde{\mu})+\tilde{r}_n(\tilde{\nu}).
$$
Similarly, the \textit{infinitesimal Boolean additive convolution of $\tilde{\mu}$ and $\tilde{\nu}$} is the infinitesimal distribution $\tilde{\mu}\uplus\tilde{\nu}\in \tilde{\DD}(k)$ uniquely determined by the fact that
$$
	\tilde{b}_n(\tilde{\mu}\uplus \tilde{\nu})=\tilde{b}_n(\tilde{\mu})+\tilde{b}_n(\tilde{\nu}).
$$
\end{definition}

\begin{remark}
In shuffle algebra, additive convolution is defined in the following way. If $\tilde{\Phi},\tilde{\Psi}\in\tilde{G}$ are the characters associated to $\tilde{\mu},\tilde{\nu}\in\tilde{\DD}(k)$, then the character $\tilde{\Phi}\boxplus\tilde{\Psi}\in \tilde{G}$ associated to $\tilde{\mu}\boxplus\tilde{\nu}\in \tilde{\DD}(k)$ is given by 
$$
	\tilde{\Phi}\boxplus\tilde{\Psi}:=\EE_\prec(\LL_\prec(\tilde{\Phi})+\LL_\prec(\tilde{\Psi})),
$$
and the the character $\tilde{\Phi}\uplus\tilde{\Psi}\in \tilde{G}$ associated to $\tilde{\mu}\uplus\tilde{\nu}\in \tilde{\DD}(k)$ is given by
$$
	\tilde{\Phi}\uplus\tilde{\Psi}:=\EE_\succ(\LL_\succ(\tilde{\Phi})+\LL_\succ(\tilde{\Psi})).
$$
\end{remark}

\noindent For $\tilde{\mu}\in\tilde{\DD}(k)$ and $s\geq 0$, we denote by $\tilde{\mu}^{\boxplus s},\tilde{\mu}^{\uplus s}\in\tilde{\DD}(k)$ the infinitesimal laws such that 
$$
	\tilde{r}_n\pai \tilde{\mu}^{\boxplus s}\pad 
	=s\tilde{r}_n(\tilde{\mu})\qquad\text{and}\qquad \tilde{b}_n\pai \tilde{\mu}^{\uplus s}\pad =s\tilde{b}_n(\tilde{\mu}).
$$

Now we are ready to prove Theorem \ref{Thm1.2}.

\begin{proof}[Proof of Theorem \ref{Thm1.2}]
Given an infinitesimal law $\tilde\mu  = (\mu,\mu^\prime)$, consider the associated linear functionals $\Phi,\Phi^\prime:H\to\mathbb{C}$ and $\tilde\Phi = \Phi + \hbar\Phi^\prime$. We will show that if 
$$
	\tilde{\mathbb{B}}_t(\tilde\Phi) = \left(\tilde\Phi^{\boxplus 1+t} \right)^{\uplus \frac{1}{1+t}},
$$
then $\tilde{\mathbb{B}}_1 = \tilde{\mathbb{B}}$ and $\tilde{\mathbb{B}}_s\circ \tilde{\mathbb{B}}_t = \tilde{\mathbb{B}}_{s+t}$, for any $s,t\geq0.$ Let  $\tilde{\kappa}= \mathcal{L}_\prec(\tilde{\Phi})$ be the $\mathbb{G}$-valued free infinitesimal cumulant character of $\tilde{\Phi}$. We can follow the proof of Lemma 42 in \cite{EP4} in order to obtain
\begin{equation}
\label{BPShuffle}
	\tilde{\mathbb{B}}_t(\tilde\Phi) = \E_\prec(t\theta_{\tilde\Phi}(\tilde\kappa))^{\uplus \frac{1}{t}}= \E_\prec\left(\theta_{\E_\prec(t\tilde\kappa)}(\tilde\kappa)\right).
\end{equation}
From above, it follows that $\tilde{\mathbb{B}}_1 = \tilde{\mathbb{B}}$. To prove the semigroup property, by \eqref{BPShuffle} we have
$$ 
	\tilde{\mathbb{B}}_s\circ \tilde{\mathbb{B}}_t(\tilde\Phi) 
	= \tilde{\mathbb{B}}_s\left( \E_\prec\left(\theta_{\E_\prec(t\tilde\kappa)}(\tilde\kappa)\right)\right) 
	= \E_\prec\left(\theta_{\E_\prec\left(s\theta_{\E_\prec(t\tilde\kappa)}(\tilde\kappa)\right)}(\theta_{\E_\prec(t\tilde\kappa)}(\tilde\kappa))\right).
$$ 
Looking at the argument of the last exponential, by the action property of $\theta$ we can get
\begin{eqnarray*}
	\tilde{\gamma}
	:=\theta_{\E_\prec\left(s\theta_{\E_\prec(t\tilde\kappa)}(\tilde\kappa)\right)}(\theta_{\E_\prec(t\tilde\kappa)}(\tilde\kappa)) 
	&=& \theta_{\E_\prec(t\tilde{\kappa})\star \E_\prec\left(\theta_{\E_\prec(t\tilde\kappa)}(s\tilde{\kappa}) \right)}(\tilde\kappa).
\end{eqnarray*}
By \eqref{shuffleadjointgroup} and \eqref{addconv},
$$
	\E_\prec(t\tilde{\kappa})\star \E_\prec\left(\theta_{\E_\prec(t\tilde\kappa)}(s\tilde{\kappa}) \right)
	= \E_\prec(t\tilde\kappa) \star\left( \E_\prec(s \tilde\kappa) \bole \E_\prec(t\tilde\kappa) \right) 
	= \E_\prec(t\tilde\kappa) \boxplus \E_\prec(s \tilde\kappa) 
	= \E_\prec((t+s)\tilde\kappa).
$$
Hence
$$ 
	\tilde{\mathbb{B}}_s\circ \tilde{\mathbb{B}}_t(\tilde\Phi)  
	= \E_\prec(\tilde{\gamma})  
	= \E_\prec(\theta_{\E_\prec((t+s)\tilde\kappa)}(\tilde\kappa)) 
	= \tilde{\mathbb{B}}_{t+s}(\Phi),
$$
where we used \eqref{BPShuffle} in the last equality.
\end{proof}

\begin{remark}
We can also do this in the usual approach. The main idea is to extend into the $\ggg$-valued case the ideas in \cite{BN2}. For instance to prove that $\tilde\bb_1(\tilde\mu)=\tilde\bb(\tilde\mu)$, we proceed like this:

Let us denote by $\tilde{\nu}_1:=\bb_1(\tilde{\mu})=\pai \tilde{\mu}^{\boxplus 2}\pad^{\uplus \frac{1}{2}}$ and $\tilde{\nu}:=\bb(\tilde{\mu})$. Now we can use \eqref{GrassBooFree} to get
$$ 
	\tilde{b}_n(\tilde{\nu}_1)
	=\frac{1}{2}  \tilde{b}_n\pai \tilde{\mu}^{\boxplus 2}\pad 
	=\frac{1}{2}\sum_{\pi\in \NN\CC_{irr}(n)} \tilde{r}_\pi \pai \tilde{\mu}^{\boxplus 2}\pad 
	=\frac{1}{2}\sum_{\pi\in \NC_{irr}(n)} 2^{|\pi|} \tilde{r}_\pi  (\tilde{\mu}).
$$
On the other hand, for all $n\geq 1$ we have $q\tilde{r}_n(\tilde{\nu})= \tilde{b}_n(\tilde{\mu})$, so we can compute
\begin{align*}
	\tilde{b}_n(\tilde{\nu})
	=\sum_{\sigma\in \NN\CC_{irr}(n)}  \tilde{r}_\sigma (\tilde{\nu}) 
	&=\sum_{\sigma\in \NN\CC_{irr}(n)}  \tilde{b}_\sigma (\tilde{\mu}) 
	=\sum_{\sigma\in \NN\CC_{irr}(n)} \prod_{V\in\sigma} \Big( \sum_{\tau\in \NN\CC_{irr}(V)}  \tilde{b}_\tau (\tilde{\mu})\Big)\\
	&=\sum_{\sigma\in \NN\CC_{irr}(n)}\sum_{\pi\ll\sigma}  \tilde{r}_\pi (\tilde{\mu})
	= \sum_{\pi\in \NN\CC_{irr}(n)}\tilde{r}_\pi (\tilde{\mu} )\sum_{\sigma \gg \pi} 1 ,
\end{align*}
where the last equality we just switched the order of the sums. {Finally, from Proposition \ref{formula.order} below, we know that $card\{ \sigma \in NC(n)|\sigma \gg \pi \}=2^{|\pi|-1}$, and by the previous analysis we get that $\tilde{b}_n(\tilde{\nu}_1)= \tilde{b}_n(\tilde{\nu})$ for all $n\geq 1$. Finally, Boolean moment-cumulant formula \eqref{GrassBooMomCum} implies that $\tilde{\varphi}_n(\tilde{\nu}_1)= \tilde{\varphi}_n(\tilde{\nu})$ for all $n\geq 1$ and we conclude that $\tilde{\nu}_1=\tilde{\nu}$.
 }
\end{remark}

\appendix

\section{Appendix}
\label{sec:appendix}


\subsection{Partitions}
\label{ssec:parititions}

Here we give a brief introduction to the definitions and notations used in this paper regarding partitions of a set. (For a broader explanation on these sets we refer the reader to \cite{AHLV})

\begin{definition}[Types of partitions]
\label{def:partitiontypes}
Let us fix a positive integer $n$. 
\begin{itemize}
\item A \textit{partition $\pi$ of $[n]:=\{1,\dots,n\}$} is a set of the form $\pi=\{V_1,\dots,V_k\}$ where $V_1,\dots, V_k$ (called \textit{blocks} of $\pi$) are pairwise disjoint non-empty subsets of $[n]$ such that $V_1\cup \dots \cup V_k=[n]$. 
\item The number of blocks of $\pi$ is denoted by $|\pi|$. 
\item We say that $\pi$ is an \textit{interval partition} if all the blocks $V\in \pi$ are of the form $V=\{i,i+1,\dots,i+j\}$ for some integers $1\leq i\leq i+j \leq n$. 
\item We say that $\pi$ is a \textit{non-crossing partition} if for every $1\leq i < j < k < l \leq n$ such that $i, k\in V_a$ (are on the same block) and $j,l\in V_b$ (are on the same block), then it necessarily follows that $a=b$ (all $i,j,k,l$ are in the same block). 
\item For two distinct blocks $V,W\in\pi$, we say that \textit{$V$ is nested inside $W$} if there exist $i,j\in W$ such that for all $v\in V$ we have $i<v<j$.
\item We say that $(\pi,\lambda)$ is a \textit{monotone partition} if $\pi$ is non-crossing and $\lambda:\pi\to \{1,\dots,|\pi|\}$ is a bijective function (an ordering of the blocks of $\pi$) such that if $V,W\in\pi$ and $V$ is nested inside $W$, then $\lambda(W)<\lambda(V)$.
\end{itemize}
We will denote by $\PP(n)$, $\NN\CC(n)$, $\II(n)$, $\MM(n)$ the sets of all, non-crossing, interval, and monotone partitions of $[n]$, respectively.
\end{definition}

Given a partition $\pi\in \NC(n)$, it is useful to know that the number of possible orders $\lambda$ on the blocks of $\pi$ such that $(\pi,\lambda)$ becomes a monotone partition is 
\begin{equation}
\label{monotone.labellings}
	m(\pi):=\frac{|\pi|!}{\tau(\pi)!},
\end{equation}
where $\tau(\pi)!$ is the tree factorial of the nesting forest of a partition $\pi$. {We refer the reader to \cite{AHLV} for a proof of this identity and further details on this topic.} 
From the previous consideration we obtain the following observation:

\begin{remark}
For any sequence of coefficients $(c_\pi)_{\pi\in \NC(n)}$ we have that
\begin{equation}
	\sum_{(\pi,\lambda)\in \MM(n)} \frac{1}{|\pi|!} c_\pi 
		= \sum_{\pi\in \NC(n)} \frac{1}{\tau(\pi)!} c_\pi.
\end{equation}
\end{remark}

\begin{definition}
	\label{order.NC}
In this note we consider two partial orders on $\NN\CC(n)$.
\begin{itemize}
\item \textit{Reversed refinement order} $\leq$. For $\pi,\sigma\in \NN\CC(n)$, we write ``$\pi \leq \sigma$'' if every block of $\sigma$ is a union of blocks of $\pi$. The maximal element of $\NN\CC(n)$ with this order is $1_n:=\{\{1,\dots,n\}\}$ (the partition of $[n]$ with only one block), and the minimal element is $0_n:=\{\{1\},\{2\},\dots,\{n\}\}$ (the partition of $[n]$ with $n$ blocks). $\NC(n)$ with this order $\leq$ actually form a lattice. For $\sigma,\pi\in \NC(n)$ we write $\text{M\"ob}(\sigma,\pi)$ to refer to the M\"obius function on $\NC(n)$. For much more detailed view this topic we refer the reader to \cite{NS}.
\item The \textit{min-max order} $\ll$. For $\pi,\sigma\in \NN\CC(n)$, we write ``$\pi \ll \sigma$'' to mean that $\pi \leq \sigma$ and that for every block $V$ of $\sigma$ there exists a block $W$ of $\pi$ such that $\min(V),\max(V) \in W$. We refer the reader to \cite{BN1} for a broader explanation on this partial order.
\end{itemize}
\end{definition}

\begin{definition}
\label{defi.irreducible.partitions}
We will say that a partition $\pi$ is \textit{irreducible} if $\pi\ll 1_n$. This is equivalent to the fact that $1$ and $n$ are in the same block of $\pi$. The set of non-crossing irreducible partitions and monotone irreducible partitions are denoted by $\NC_{irr}(n)$ and $\MM_{irr}(n)$, respectively.
\end{definition}

It is useful to recall the following property of the min-max order, that is a particular case of Proposition 2.13 in \cite{BN1}:

\begin{proposition}
	\label{formula.order}
Let $\pi\in \NC_{irr}(n)$ and $p$ an integer with $1\leq p\leq |\pi|$. Then
$$
	card\{ \sigma \in \NC(n)|\sigma \gg \pi \text{ and } |\sigma|=p\}=\binom{|\pi|-1}{p-1}.
$$
\end{proposition}

{
In particular, we have the following formula
\begin{equation}
	card\{ \sigma \in \NC(n)|\sigma \gg \pi \}
		=\sum_{p=1}^{|\pi|} \binom{|\pi|-1}{p-1}=  2^{|\pi|-1}.
\end{equation}}

{
When writing the monotone cumulants in terms of free or Boolean cumulants, it was recently shown in \cite{Magnus2021} that the coefficients are dictated by Murua's $\omega$ function.

\begin{definition}[Murua's $\omega$]
\label{def.murua.coeff}
Let $\pi\in \NC_{irr}(n)$. We will write $\omega_k(\pi)$ for the number of increasing $k$-colored non-crossing partitions, namely, the number of ways one can decorate the blocks $V\in\pi$ with a color $f(V)\in [k]$ in such a way that if the block $V$ is nested inside the block $W$, this implies $f(V)<f(W)$.  Then, for every $\pi\in\NC_{irr}(n)$  we define
\begin{equation}
\label{eq.muruaomega}
	\omega(\pi):=\sum\limits_{k=1}^n\frac{(-1)^{k+1}}{k}\omega_k(\pi).
\end{equation}
\end{definition}
}



\subsection{Independences}
\label{ssec:independences}

Different notions of independence play an important role in non-commutative probability. Muraki \cite{Mur2} proved that there are only five natural notions of independence: tensor, free, Boolean, monotone and anti-monotone. We give the precise definitions of the independences used in this paper.

\begin{definition} (Independences) \label{def:independences}
Let $(\A,\varphi)$ be a non-commutative probability space. Consider an index set $I$ and $\A_1,\ldots,\A_k$ subalgebras of $\A$.
\begin{enumerate}
    \item We say that $\A_1,\ldots,\A_k$ are \textit{freely independent} if each $\A_i$ is unital and 
$$
	\varphi(a_1\cdots a_n) = 0
$$ 
whenever $n\geq1$, $a_1\in \A_{i_1},\ldots,a_n\in \A_{i_n}$,  $i_j \neq i_{j+1}$ for $1\leq j< n$ and $\varphi(a_1) = \cdots =\varphi(a_n) = 0$.

    \item We say that $\A_1,\ldots,\A_k$ are \textit{Boolean independent} if
$$
    	\varphi(a_1\cdots a_n) = \varphi(a_1)\cdots\varphi(a_n)
$$ 
whenever $n\geq1$, $a_1\in \A_{i_1},\ldots,a_n\in \A_{i_n}$ and $i_j \neq i_{j+1}$ for $1\leq j< n$.

    \item We say that $\A_1,\ldots,\A_k$ are \textit{monotone independent} if
$$
    	\varphi(a_1\cdots a_\ell \cdots a_n) = \varphi(a_\ell)\varphi(a_1\cdots a_{\ell-1}a_{\ell+1}\cdots a_n)
$$ 
whenever $n\geq1$, $a_1\in \A_{i_1},\ldots,a_n\in \A_{i_n}$,  $i_j \neq i_{j+1}$ for $1\leq j< n$, $i_{\ell-1}< i_\ell$ and $i_\ell>i_{\ell+1}.$
\end{enumerate}
\end{definition}

Speicher introduced free cumulants in order to characterise Voiculescu's free independence. More precisely, he proved that free independence is equivalent to the condition of vanishing mixed free cumulants. Boolean cumulants satisfy an analogous characterisation for Boolean independence. In this spirit, Février and Nica introduced the notion of infinitesimal free cumulants such that the condition of vanishing mixed cumulants is equivalent to the notion of infinitesimal free independence. For completeness, we provide the definition of infinitesimal freeness.

\begin{definition} (Infinitesimal freeness)
Let $(\A,\varphi,\varphi^\prime)$ be an incps. Let $\A_1,\ldots,\A_k \subset \A$ be unital subalgebras of $\A$. We say that $\A_1,\ldots,\A_k$ are \textit{infinitesimally free} if the following holds: for each $n\geq1$ and any sequence of indices $i_1,\ldots,i_n\in \{1,\ldots,k\}$ such that $i_j\neq i_{j+1}$, for $1\leq j<n$, and elements $a_1\in \A_{i_1},\ldots,a_n\in \A_{i_n}$ such that $\varphi(a_j) = 0$, for $j=1,\ldots,n$, the following is satisfied

\begin{enumerate}
\item 
$\varphi(a_1 \cdot_{\!\scriptscriptstyle{\A}} \cdots \cdot_{\!\scriptscriptstyle{\A}} a_n)=0.$

\item 
$\varphi^\prime (a_1 \cdot_{\!\scriptscriptstyle{\A}} \cdots \cdot_{\!\scriptscriptstyle{\A}} a_n) = \varphi(a_1 \cdot_{\!\scriptscriptstyle{\A}} a_n)\varphi(a_2 \cdot_{\!\scriptscriptstyle{\A}} a_{n-1})\cdots\varphi(a_{(n-1)/2} \cdot_{\!\scriptscriptstyle{\A}} a_{(n+3)/2}) \varphi^\prime(a_{(n+1)/2})$
if $n$ is odd and $i_1=i_n$, $i_2=i_{n-1},\ldots,i_{(n-1)/2}=i_{(n+3)/2}$, and $\varphi^\prime (a_1 \cdot_{\!\scriptscriptstyle{\A}} \cdots \cdot_{\!\scriptscriptstyle{\A}} a_n) =0$ otherwise.
\end{enumerate}
\end{definition}


{One could use the definition of infinitesimal Boolean cumulants in order to obtain the notion of infinitesimal Boolean independence.} Let $(\A,\varphi,\varphi^\prime)$ be an incps and $\A_1,\ldots,\A_k$ be subalgebras of $\A$. Consider also the $\mathbb{G}$-valued Boolean cumulants $\{\tilde{b}_n \colon  \A\to\mathbb{G}\}_{n\geq1}$. 
Then, we assume the vanishing mixed cumulants condition for $\{\tilde{b}_n\colon \A\to\mathbb{G}\}_{n\geq1}$, i.e.,
$$
	\tilde{b}_n(a_1,\ldots,a_n)=0
$$
when $n\geq2$, and if there exists $1\leq s<r\leq n$ such that $a_r\in \A_{i(r)}, a_s\in \A_{i(s)}$ and $i(r)\neq i(s)$. 
\par We want to find some conditions for $\varphi$ and $\varphi^\prime$, which are equivalent to the vanishing of mixed cumulants. 
Let $n\geq1$ and take elements $a_1\in \A_{i(1)},\ldots,a_n\in \A_{i(n)}$, where $i_1,\ldots,i_n\in \{1,\ldots, k\}$ and $i_j\neq i_{j+1}$ for $1\leq j<n$. Since we are assuming the vanishing mixed cumulants condition, the unique partition $\pi$ which does not give a zero contribution in the sum
\begin{equation}
\label{BTilde}
	\tilde{\varphi}(a_1 \cdot_{\!\scriptscriptstyle{\A}} \cdots \cdot_{\!\scriptscriptstyle{\A}} a_n) 
	= \sum_{\pi\in \mathcal{I}(n)} \tilde{\beta}_\pi(a_1,\ldots,a_n),
\end{equation}
is $\pi = \{\{1\},\{2\},\ldots,\{n\}\}.$ Hence
$$ 
	\tilde{\varphi}(a_1 \cdot_{\!\scriptscriptstyle{\A}} \cdots \cdot_{\!\scriptscriptstyle{\A}} a_n) 
	= \sum_{\pi\in \mathcal{I}(n)} \tilde{\beta}_\pi(a_1,\ldots,a_n) 
	= \tilde{\varphi}(a_1)\cdots \tilde{\varphi}(a_n).
$$
Recalling that $\tilde{\varphi} = \varphi + \hbar \varphi^\prime$, we have that $ \tilde{\varphi}(a_1 \cdot_{\!\scriptscriptstyle{\A}} \cdots \cdot_{\!\scriptscriptstyle{\A}} a_n) =  \tilde{\varphi}(a_1)\cdots \tilde{\varphi}(a_n)$ is equivalent to
\allowdisplaybreaks
\begin{eqnarray}
\label{IBI1}
	\varphi(a_1 \cdot_{\!\scriptscriptstyle{\A}} \cdots \cdot_{\!\scriptscriptstyle{\A}} a_n) 
	&=& \varphi(a_1)\cdots\varphi (a_n),\\
\label{IBI2}
	\varphi^\prime(a_1 \cdot_{\!\scriptscriptstyle{\A}} \cdots \cdot_{\!\scriptscriptstyle{\A}} a_n) 
	&=& \sum_{m=1}^n \varphi^\prime(a_m) \prod_{1\leq k\leq n \atop k\neq m} \varphi(a_k).
\end{eqnarray}

The above equations provide the conditions for infinitesimal Boolean independence.

\begin{definition}
Let $(\A,\varphi,\varphi^\prime)$ be an incps. Consider $\A_1,\ldots, \A_k$ subalgebras of $\A$. We say that $\A_1,\ldots\,\A_k$ are \textit{infinitesimally Boolean independent} if for each $n\geq1$ and any sequence of indices $i_1,\ldots,i_n\in \{1,\ldots,k\}$ such that $i_j\neq i_{j+1}$ for $j=1,\ldots,n-1$, and elements $a_1\in \A_{i_1},\ldots, a_n\in \A_{i_n}$, we have that \eqref{IBI1} and \eqref{IBI2} hold. 
\end{definition}

We note that condition \eqref{IBI1} simply refers to the usual Boolean independence.

{
\begin{remark}
The above definition coincides with the first order Boolean differential independence introduced in \cite{Has}.
\end{remark}
}

\begin{remark}
	It is easy to prove that infinitesimal Boolean independence recently defined actually implies the conditions of vanishing mixed Boolean and infinitesimal Boolean cumulants. One can obtain a proof of this fact by following the usual case of Boolean independence.
\end{remark}



\bibliographystyle{alpha}
\bibliography{Infinitesimal230821}	

\begin{thebibliography}{CEFPP21}

\bibitem[AHLV15]{AHLV}
Octavio Arizmendi, Takahiro Hasebe, Franz Lehner, and Carlos Vargas.
\newblock Relations between cumulants in noncommutative probability.
\newblock {\em Advances in Mathematics}, 282:56--92, 2015.

\bibitem[BGN03]{BGN}
Philippe Biane, Frederick Goodman, and Alexandru Nica.
\newblock Non-crossing cumulants of type {B}.
\newblock {\em Transactions of the American Mathematical Society},
  355(6):2263--2303, 2003.

\bibitem[BN08a]{BN1}
Serban Belinschi and Alexandru Nica.
\newblock $\eta$-series and a boolean bercovici--pata bijection for bounded
  k-tuples.
\newblock {\em Advances in Mathematics}, 217(1):1--41, 2008.

\bibitem[BN08b]{BN2}
Serban Belinschi and Alexandru Nica.
\newblock On a remarkable semigroup of homomorphisms with respect to free
  multiplicative convolution.
\newblock {\em Indiana university mathematics journal}, pages 1679--1713, 2008.

\bibitem[BN09]{BN3}
Serban Belinschi and Alexandru Nica.
\newblock Free brownian motion and evolution towards $\boxplus$-infinite
  divisibility for k-tuples.
\newblock {\em International Journal of Mathematics}, 20(03):309--338, 2009.

\bibitem[BS12]{BS}
Serban Belinschi and Dimitri Shlyakhtenko.
\newblock Free probability of type {B}: analytic interpretation and
  applications.
\newblock {\em American Journal of Mathematics}, 134(1):193--234, 2012.

\bibitem[CEFPP21]{Magnus2021}
Adrian Celestino, Kurusch Ebrahimi-Fard, Frederic Patras, and Daniel Perales.
\newblock Cumulant-cumulant relations in free probability theory from {M}agnus'
  expansion.
\newblock {\em To appear in Foundations of Computational Mathematics}, 2021.

\bibitem[CHS15]{CHS}
Benoit Collins, Takahiro Hasebe, and Noriyoshi Sakuma.
\newblock Free probability for purely discrete eigenvalues of random matrices.
\newblock {\em J. Math. Soc. Japan}, 70(3):1111--1150, 2015.

\bibitem[EFM09]{EM}
Kurusch Ebrahimi-Fard and Dominique Manchon.
\newblock Dendriform equations.
\newblock {\em Journal of Algebra}, 322(11):4053--4079, 2009.

\bibitem[EFP15]{EP1}
Kurusch Ebrahimi-Fard and Fr{\'e}d{\'e}ric Patras.
\newblock Cumulants, free cumulants and half-shuffles.
\newblock {\em Proceedings of the Royal Society A: Mathematical, Physical and
  Engineering Sciences}, 471(2176):20140843, 2015.

\bibitem[EFP18]{EP2}
Kurusch Ebrahimi-Fard and Fr{\'e}d{\'e}ric Patras.
\newblock Monotone, free, and boolean cumulants: a shuffle algebra approach.
\newblock {\em Advances in Mathematics}, 328:112--132, 2018.

\bibitem[EFP19]{EP4}
Kurusch Ebrahimi-Fard and Fr{\'e}d{\'e}ric Patras.
\newblock Shuffle group laws: applications in free probability.
\newblock {\em Proceedings of the London Mathematical Society},
  119(3):814--840, 2019.

\bibitem[EFP20]{EP3}
Kurusch Ebrahimi-Fard and Fr{\'e}d{\'e}ric Patras.
\newblock A group-theoretical approach to conditionally free cumulants.
\newblock {\em IRMA Lectures in Mathematics and Theoretical Physics},
  31:67--92, 2020.

\bibitem[Fev12]{Fev}
Maxime Fevrier.
\newblock Higher order infinitesimal freeness.
\newblock {\em Indiana University Mathematics Journal}, pages 249--295, 2012.

\bibitem[FN10]{FN}
Maxime F{\'e}vrier and Alexandru Nica.
\newblock Infinitesimal non-crossing cumulants and free probability of type
  {B}.
\newblock {\em Journal of Functional Analysis}, 258(9):2983--3023, 2010.

\bibitem[Foi07]{Foi}
Lo{\"\i}c Foissy.
\newblock Bidendriform bialgebras, trees, and free quasi-symmetric functions.
\newblock {\em Journal of Pure and Applied Algebra}, 209(2):439--459, 2007.

\bibitem[Has11]{Has}
Takahiro Hasebe.
\newblock Differential independence via an associative product of infinitely
  many linear functionals.
\newblock In {\em Colloq. Math}, volume 124, pages 79--94, 2011.

\bibitem[HS11a]{HS2}
Takahiro Hasebe and Hayato Saigo.
\newblock Joint cumulants for natural independence.
\newblock {\em Electronic Communications in Probability}, 16:491--506, 2011.

\bibitem[HS11b]{HS1}
Takahiro Hasebe and Hayato Saigo.
\newblock The monotone cumulants.
\newblock In {\em Annales de l'institut Henri Poincar{\'e} - Probabilit\'es et
  Statistiques}, volume~47, pages 1160--1170, 2011.

\bibitem[Len07]{Len}
Romuald Lenczewski.
\newblock Decompositions of the free additive convolution.
\newblock {\em Journal of Functional Analysis}, 246(2):330--365, 2007.

\bibitem[Min18]{Min}
James~A Mingo.
\newblock Non-crossing annular pairings and the infinitesimal distribution of
  the {GOE}.
\newblock {\em Journal of the London Mathematical Society}, 2018.

\bibitem[Mur00]{Mur1}
Naofumi Muraki.
\newblock Monotonic convolution and monotonic {L}{\'e}vy--{H}incin formula.
\newblock {\em preprint}, 2000.

\bibitem[Mur02]{Mur2}
Naofumi Muraki.
\newblock The five independences as quasi-universal products.
\newblock {\em Infinite Dimensional Analysis, Quantum Probability and Related
  Topics}, 5(01):113--134, 2002.

\bibitem[NS06]{NS}
Alexandru Nica and Roland Speicher.
\newblock {\em Lectures on the combinatorics of free probability}, volume~13.
\newblock Cambridge University Press, 2006.

\bibitem[PT20]{perales2020operator}
Daniel Perales and Pei-Lun Tseng.
\newblock On operator-valued infinitesimal boolean and monotone independence.
\newblock {\em arXiv preprint arXiv:2010.15286}, 2020.

\bibitem[Reu93]{Reu}
Christophe Reutenauer.
\newblock {\em Free Lie algebras}.
\newblock Oxford University Press, 1993.

\bibitem[Shl18]{Shl}
Dimitri Shlyakhtenko.
\newblock Free probability of type {B} and asymptotics of finite-rank
  perturbations of random matrices.
\newblock {\em Indiana Univ. Math. J.}, 67:971--991, 2018.

\bibitem[Spe94]{Spe}
Roland Speicher.
\newblock Multiplicative functions on the lattice of non-crossing partitions
  and free convolution.
\newblock {\em Mathematische Annalen}, 298(1):611--628, 1994.

\bibitem[SW97]{SW}
Roland Speicher and Reza Woroudi.
\newblock Boolean convolution.
\newblock {\em Fields Institute Communications}, 12:267--279, 1997.

\bibitem[Swe69]{Sweedler69}
Moss~E. Sweedler.
\newblock {\em Hopf Algebras}.
\newblock W. A. Benjamin, 1969.

\bibitem[Voi85]{Voi}
Dan Voiculescu.
\newblock Symmetries of some reduced free product $\mathcal{C}^*$-algebras.
\newblock In {\em Operator algebras and their connections with topology and
  ergodic theory}, pages 556--588. Springer, 1985.

\end{thebibliography}

\end{document}